\documentclass[a4paper]{article}

\usepackage[english]{babel}
\usepackage[utf8]{inputenc}
\usepackage{amsmath}
\usepackage{amssymb}
\usepackage{graphicx}
\usepackage{amsthm}
\usepackage{verbatim}
\usepackage{authblk}
\usepackage{todonotes}
\usepackage{subfig}
\usepackage{wrapfig}
\usepackage{hyperref}

\newtheorem{theorem}{Theorem}
\newtheorem{lemma}{Lemma}
\newtheorem{proposition}{Proposition}

\newtheorem{claim}{Claim}

\newcommand\ceil[1]{\lceil#1\rceil}
\title{Graph fractal dimension and structure of fractal networks: a combinatorial perspective}

\author[1,3]{Pavel Skums}
\author[2]{Leonid Bunimovich}
\affil[1]{Department of Computer Science, Georgia State University, 1 Park Pl NE, Atlanta, GA, USA, 30303}
\affil[2]{School of Mathematics, Georgia Institute of Technology, 686 Cherry St NW, Atlanta, GA, USA, 30313}
\affil[3]{Corresponding author. \textit {Email: pskums@gsu.edu}}

\date{\today}

\begin{document}
\maketitle

\begin{abstract}
In this paper we study self-similar and fractal networks from the combinatorial perspective.  We establish analogues of topological (Lebesgue) and fractal (Hausdorff) dimensions for graphs and demonstrate that they are
naturally related to known graph-theoretical characteristics: rank dimension and product (or Prague or Ne{\v s}et{\v r}il-R{\"o}dl) dimension. Our approach reveals how self-similarity and fractality of a network are defined by a pattern of overlaps between densely connected network communities. It allows us to identify fractal graphs, explore the relations between graph fractality, graph colorings and graph Kolmogorov complexity, and analyze the fractality of several classes of graphs and network models, as well as of a number of real-life networks. We demonstrate the application of our framework to evolutionary studies by revealing the growth of self-organization of heterogeneous viral populations over the course of their intra-host evolution, thus suggesting mechanisms of their gradual adaptation to the host's environment. As far as the authors know, the proposed approach is the first theoretical framework for study of network fractality within the combinatorial paradigm. The obtained results lay a foundation for studying fractal properties of complex networks using combinatorial methods and algorithms. 

% We also demonstrate applications of this approach by establishing a novel property of general compact metric spaces using ideas from hypergraphs theory and by proving an estimation for Prague dimension of almost all graphs using methods from algorithmic information theory.

\medskip

\end{abstract}

\section{Introduction}

Fractals are geometric objects that are widespread in nature and appear in many research domains, including dynamical systems, physics, biology and behavioural sciences \cite{falconer2004fractal}. By Mandelbrot's classical definition, geometric fractal is a topological space (usually a subspace of an Euclidean space), whose topological (Lebesgue) dimension is strictly smaller than the fractal (Hausdorff) dimension. It is also usually assumed that fractals have some form of geometric or statistical self-similarity \cite{falconer2004fractal}.

 Lately there was a growing interest in studying self-similarity and fractal
properties of complex networks, which is largely inspired by applications in biology, sociology, chemistry and computer science \cite{song2005self,shanker2007defining,dorogovtsev2013evolution,newman2003structure}. Although such studies are usually based on geniune  ideas from graph theory and general topology and provided a deep insight into structures of complex networks and mechanisms of their formation, they are often not supported by a rigorous mathematical framework.
% approach the problems under consideration in a rigorous mathematical way 
As a result, such methods may not be directly applicable to many important classes of graphs and networks \cite{li2005towards,willinger2009mathematics}. In particular, many studies translate the definition of a topological fractal to networks by considering a graph as the finite metric space with the metric being the standard shortest path length, and identifying graph fractal dimension with the Minkowski–Bouligand (box-counting) dimension \cite{song2005self,shanker2007defining}. However, direct applications of the continuous definition to discrete objects such as networks can be problematic. Indeed, under this definition many real-life networks do not have well-defined fractal dimension and/or are not fractal and self-similar. This is in particular due to the fact that these networks have so-called ‘small-world’ property, which implies that their diameters are exponentially smaller than the numbers of their vertices \cite{song2005self}. Moreover, even if the box-counting dimension of a  network can be defined and calculated, it is challenging to associate it with graph structural/topological properties. As regards to the phenomenon of network self-similarity, previous studies described it as the preservation of network properties under a length-scale transformation \cite{song2005self}. However, geometric fractals possess somewhat stronger property: they are {\it comprised} of parts topologically similar to the whole rather than just have similar features at different scales. Finally, many computational tasks associated with the continuous definitions cannot be formulated as well-defined algorithmic problems and studied within the framework of theory of computational complexity, discrete optimization and machine learning. Thus, it is highly desirable to develop an understanding of graph dimensionsionality, self-similarity and fractality based on innate ideas and machineries of graph theory and combinatorics. There are several studies that translate certain notions of topological dimension theory to graphs using combinatorial methods \cite{smyth2010topological,evako1994dimension}.  However, to the best of our knowledge, a rigorous combinatorial theory of graph-theoretical analogues of topological fractals still has not been developed. 

In this paper we propose a combinatorial approach to the fractality of graphs, which consider natural network analogues of Lebesgue and Hausdorff dimensions of topological spaces from the graph-theoretical point of view. This approach allows to overcome the aforementioned difficulties and provides mathematically rigorous, algorithmically tractable and practical framework for study of network self-similarity and fractality. Roughly speaking, our approach suggests that fractality of a network is more naturally related to a pattern of overlaps between densely connected network communities rather than to the distances between individual nodes. It is worth noting that overlapping community structure of complex networks received considerable attention in network theory and has been a subject of multiple studies \cite{palla2005uncovering,ahn2010link}. Furthermore, such approach allows us to exploit the duality between partitions of networks into communities and encoding of networks using set systems. This duality has been studied in graph theory for a long time \cite{berge1984hypergraphs}, and allows for topological and information-theoretical interpretations of network self-similarity and fractality. 

The major results of this study can be summarized as follows: 

1) Lebesgue and Hausdorff dimensions of graphs are naturally related to known characteristics from the graph theory and combinatorics: rank dimension \cite{berge1984hypergraphs} and product (or Prague or Ne{\v s}et{\v r}il-R{\"o}dl) dimension  \cite{hell2004graphs}. These dimensions are associated with the patterns of overlapping cliques in graphs.  We underpin the connection between general topological dimensions and their network analogues by demonstrating that they measure the analogous characteristics of the respective objects:

\begin{itemize}
    \item Topological Lebesgue dimension and graph rank dimension are both associated with the representation of general compact metric spaces and graphs by intersecting families of sets. Such representations have been extensively studied in graph theory \cite{berge1984hypergraphs}, where it has been shown that any graph of a given rank dimension encodes the pattern of intersections of a family of finite sets with particular properties.  It turned out, that general compact metric spaces of a given Lebesgue dimension also can be approximated by intersecting families of sets with analogous properties. 
    \item Rank dimension defines a decomposition of a graph into its own images under stronger versions of graph homomorphisms, and can be interpreted as a measure of a graph self-similarity. 
\end{itemize}

% For practical purposes, another graph-theoretical characterization of Hausdorff dimension is more useful and visualizable. It is formulated in terms of covering of a graph by multi-colored cliques. 

2) Fractal graphs naturally emerge as graphs whose Lebesgue dimensions are strictly smaller than Hausdorff dimensions. We analyze in detail the fractality and self-similarity of scale-free networks,  Erd{\"o}s-Renyi graphs  and cubic and subcubic graphs. For such graphs, fractality is closely related to edge colorings, and separation of graphs into fractals and non-fractals could be considered as a generalization of one of the most renowned dichotomies in graph theory - the separation of graphs into class 1 and class 2 \cite{vizing1964estimate} (i.e graphs whose edge chromatic number is equal to $\Delta$  or $\Delta + 1$, where $\Delta$ is the maximum vertex degree of a graph). One of the examples of graph fractals is the remarkable class of  {\it snarks} \cite{gardner1977mathematical,chladny2010factorisation}. Snarks turned out to be the basic cubic fractals, with other cubic fractals being topologically reducible to them. 

3) Lebesgue and Hausdorff dimension of graphs are related to their Kolmogorov complexity -- one of the basic concepts of information theory, which is often studied in association with fractal and chaotic systems \cite{li2009Kolmogorov}. These dimensions measure the complexity of graph encoding using so-called set and vector representations. Non-fractal graphs are the graphs for which these representations are equivalent, while fractal graphs possesses additional structural properties that manifest themselves in extra dimensions needed to describe them using the latter representation.

4) Analytical estimations and experimental results reveal high self-similarity of sparse Erd{\"o}s-Renyi and Wattz-Strogatz networks, and lower self-similarity of preferential attachment and dense Erd{\"o}s-Renyi networks. Numerical experiments suggest that fractality is a rare phenomenon for basic network models, but could be significantly more common for real networks.

5) The proposed theory can be used to infer information about the mechanisms of real-life network formation. As an example, we analyzed genetic networks representing structures of 323 intra-host Hepatitis C populations sampled at different infection stages. The analysis revealed the increase of network self-similarity over the course of infection, thus suggesting intra-host viral adaptation and emergence of self-organization of viral populations over the course of their evolution.

% This relation also further strengthen the analogy between graph and topological Hausdorff dimensions, since similar relations has been previously observed for other mathematical objects \cite{ryabko1994complexity,staiger1993kolmogorov}.  

We expect that the theory of graph fractals explored in this paper will facilitate study of fractal properties of graphs and complex networks. One of its possible applications is the possibility to represent various  computational tasks as algorithmic problems to be studied within the framework of theory of algorithms and computational complexity.

\section{Basic definitions and facts from measure theory, dimension theory and graph theory}

Let $X$ be a compact metric space. A family $\mathcal{C} = \{C_{\alpha} : \alpha \in A\}$ of open subsets of $X$ is a {\it cover}, if $X = \bigcup_{\alpha \in A} C_{\alpha}$. A cover $\mathcal{C}$ is {\it $k$-cover}, if every $x\in X$ belongs to at most $k$ sets from $\mathcal{C}$; {\it $\epsilon$-cover}, if for every set $C_i\in \mathcal{C}$ its diameter $diam(C_i)$ does not exceed $\epsilon$; $(\epsilon, k)$-cover, if it is both $\epsilon$-cover and $k$-cover. {\it Lebesgue dimension} $dim_L(X)$ of $X$ is the minimal integer $k$ such that for every $\epsilon > 0$ there exists $(\epsilon, k+1)$-cover of $X$. 

Let $\mathcal{F}$ be a semiring of subsets of a set X. A function $m:\mathcal{F}\rightarrow \mathbb{R}^+_0$ is {\it a measure}, if $m(\emptyset) = 0$ and for any countable collection of pairwise disjoint sets $A_i \in \mathcal{F},i=1,...,\infty$, one has $m(\cup_{i=1}^{\infty} A_i) = \sum_{i=1}^{\infty} m(A_i)$. 

Let now $X$ be a subspace of an Euclidean space $\mathbb{R}^d$. {\it Hyper-rectangle} $R$ is a Cartesian product of semi-open intervals: $R = [a_1,b_1)\times\dots\times [a_d,b_d)$, where $a_i,b_i \in \mathbb{R}$; the {\it volume} of a the hyper-rectangle $R$ is equal to $vol(R) = \prod_{i=1}^d (b_i - a_i)$. The {\it $d$-dimensional Jordan measure} of the set $X$ is defined as
$\mathcal{J}^d(X) = \inf\{\sum_{R\in \mathcal{C}} vol(R)\},$
where the infimum is taken over all finite covers $\mathcal{C}$ of $X$ by disjoint hyper-rectangles. The {\it $d$-dimensional Lebesgue measure} of a measurable set $\mathcal{L}^d(X)$ is defined analogously, with the infimum  taken over all countable covers $\mathcal{C}$ of $X$ by (not necessarily disjoint) hyper-rectangles. Finally, the {\it $d$-dimensional Hausdorff measure} of the set $X$ is defined as $\mathcal{H}^d(X) = \lim_{\epsilon \rightarrow 0} \mathcal{H}^s_{\epsilon}(X)$, where $\mathcal{H}^d_{\epsilon}(X) = \inf\{\sum_{C\in \mathcal{C}}diam(C)^d\},$ and the infimum is taken over all $\epsilon$-covers of $X$. These 3 measures are related: the Jordan and Lebesgue measures of the set $X$ are equal, if the former exists, while Lebesgue and Hausdorff measures of Borel sets differ only by a multiplicative constant.

{\it Hausdorff dimension} $dim_H(X)$ of the set $X$ is the value

\begin{equation}\label{hdimtop}
dim_H(X) = \inf\{s \geq 0 :\mathcal{H}^s(X) < \infty\}.
\end{equation}

Lebesgue and Hausdorff dimension of $X$ are related as follows:

\begin{equation}\label{lebvshaus}
dim_L(X) \leq dim_H(X).
\end{equation}

\noindent
The set $X$ is {\it a fractal} (by Mandelbrot's definition) \cite{edgar2007measure}, if the inequality (\ref{lebvshaus}) is strict.

Now let $G=(V(G),E(G))$ be a simple graph. The notation $x\sim y$ indicates that the vertices $x,y\in V(G)$ are adjacent, and $\Delta(G)$ denotes the maximum vertex degree of $G$. We denote by $\overline{G}$ the {\it complement} of $G$, i.e. the graph on the same vertex set and with two vertices being adjacent whenever they are not adjacent in $G$. Connected components of $\overline{G}$ are called {\it co-connected components} of $G$. A graph is biconnected, if there is no vertex or edge (called {\it a bridge}), whose removal makes it disconnected.  

A graph $H$ is a \emph{subgraph} of $G$, if $V(H) \subseteq V(G)$ and $E(H) \subseteq E(G)$. A subgraph $G[U]$ is {\it induced} by a vertex subset $U\subseteq V(G)$, if it contains all edges with both endpoints in $U$. The complete graph, the chordless path and the chordless cycle on $n$ vertices are denoted by $K_n$, $P_n$ and $C_n$, respectively. A \emph{star} $K_{1, n}$ is the graph on $n+1$ vertices with one vertex of the degree $n$ and $n$ vertices of the degree 1. 

A \emph{clique} of $G$ is a set of pairwise adjacent vertices. A {\it clique number} $\omega(G)$ is the number of vertices in the largest clique of $G$. The family of cliques $\mathcal{C} = (C_1,...,C_m)$ of $G$ is a {\it clique cover}, if every edge $uv\in E(G)$ is contained in at least one clique from $\mathcal{C}$. The subgraphs forming the cover are referred to as its {\it clusters}. A cover $\mathcal{C}$ is {\it $k$-cover}, if every vertex $v\in V(G)$ belongs to at most $k$ clusters. A cluster $C\in \mathcal{C}$ {\it separates} vertices $u,v\in V(G)$, if $|C\cap \{u,v\}| = 1$. A cover is {\it separating}, if every two distinct vertices are separated by some cluster. 

Now consider a hypergraph $\mathcal{H} = (\mathcal{V}(\mathcal{H}),\mathcal{E}(\mathcal{H}))$ (i.e. a finite set $\mathcal{V}(\mathcal{H})$ together with a family of its subsets $\mathcal{E}(\mathcal{H})$ called {\it edges}). Simple graphs are special cases of hypergraphs. The  {\it rank} $r(\mathcal{H})$ is the maximal size of an edge of $\mathcal{H}$. A hypergraph $\mathcal{H}$ is {\it strongly $k$-colorable}, if one can assign colors from the set $\{1,...,k\}$ to its vertices in such a way that vertices of every edge receive different colors. The vertices of the same color form a {\it color class}. Strongly $2$-colorable simple graphs are called {\it bipartite}.  The {\it edge $k$-coloring} and {\it edge color classes} of a hypergraph are defined analogously, with the condition that the edges that share a vertex receive different colors. {\it Chromatic number} $\chi(H)$ and {\it edge chromatic number} $\chi'(H)$ are minimal numbers of colors required to color vertices and edges of a hypergraph, respectively.

Intersection graph $L = L(\mathcal{H})$ of a hypergraph $\mathcal{H}$ is a simple graph with a vertex set $V(L) = \{v_E : E\in \mathcal{E}(\mathcal{H})\}$ in a bijective correspondence with the edge set of $\mathcal{H}$ and two distinct vertices $v_E,v_F\in V(L)$ being adjacent, if and only if $E\cap F \ne \emptyset$. The following theorem establishes a connection between intersection graphs and clique $k$-covers:

\begin{theorem}\cite{berge1984hypergraphs}\label{thm:covervsinter}
	A graph $G$ is an intersection graph of a hypergraph of rank $\leq k$ if and only if it has a clique $k$-cover.
\end{theorem}

{\it Rank dimension} \cite{metelsky2003} $dim_R(G)$ of a graph $G$ is the minimal $k$ such that $G$ satisfies Theorem \ref{thm:covervsinter}.  In particular, graphs with $dim_R(G)=1$ are disjoint unions of cliques (such graphs are called {\it equivalence graphs} \cite{alon1986covering} or {\it $M$-graphs} \cite{tyshkevich1989matr}).

%% survey previous studies of krausz dimension

{\it Categorical product} of graphs $G_1$ and $G_2$ is the graph $G_1 \times G_2$ with the vertex set $V(G_1\times G_2) = V(G_1)\times V(G_2)$ with two vertices $(u_1,u_2)$ and $(v_1,v_2)$ being adjacent whenever $u_1v_1\in E(G_1)$ and $u_2v_2\in E(G_2)$. {\it Product dimension} (or {\it Prague dimension} or {\it or Ne{\v s}et{\v r}il-R{\"o}dl} in different sources) $dim_P(G)$ is the minimal integer $d$ such that $G$ is an induced subgraph of a categorical product of $d$ complete graphs \cite{hell2004graphs}.

{\it Equivalent $k$-cover} of the graph $G$ is a cover of its edges by equivalence graphs. It can be equivalently defined as a clique cover $\mathcal{C}$ such that the hypergraph $\mathcal{H}(\mathcal{C}) = (V(G),\mathcal{C})$ is edge $k$-colorable. Relations between product dimension, clique covers and intersection graphs are described by the following theorem that comprises results obtained in several prior studies:

\begin{theorem}\cite{hell2004graphs,babaits1996kmern}\label{thm:pdimcharact}
	The following statements are equivalent:
	
	1) $dim_P(\overline{G}) \leq k$;
	
	2) there exists a separating equivalent $k$-cover of $G$;
	
	3) $G$ is an intersection graph of strongly $k$-colorable hypergraph without multiple edges;
    
    4) there exists an injective mapping $\phi: V(G) \rightarrow \mathbb{N}^k$, $v \mapsto (\phi_1(v),\dots,\phi_k(v))$ such that $uv \in E(G)$ whenever $\phi_j(u) = \phi_j(v)$ for some $j\in \{1,...,k\}$.
	
\end{theorem}

\section{Lebesgue dimension of graphs}\label{lebgraph}

Lebesgue dimension of a metric space is defined through $k$-covers by sets of arbitrary small diameter. It is natural to transfer this definition to graphs using graph $k$-covers by subgraphs of smallest possible diameter, i.e. by cliques. Thus in light of Theorem \ref{thm:covervsinter} we define Lebesgue dimension of a graph through its rank dimension: 

\begin{equation}\label{kdimeqleb}
dim_L(G) = dim_R(G)-1.
\end{equation}

\noindent
An analogy between Lebesgue dimension of a metric space and rank dimension of a graph is reinforced by Theorem \ref{topspacehyper}. This theorem basically extends the analogy from graph theory back to general topology by stating that any compact metric spaces of bounded Lebesgue measure could be approximated by intersection graphs of (infinite) hypergraphs of bounded rank. To prove it, we will use the following fact:

\begin{lemma}\label{lebnumber}\cite{edgar2007measure}
Let $X$ be a compact metric space and $\mathcal{U}$ be its open cover. Then there exists $\delta>0$ (called a {\it Lebesgue number} of $\mathcal{U}$) such that for every subset $A\subseteq X$ with $diam(A) < \delta$ there is a set $U\in \mathcal{U}$ such that $A\subseteq U$.
\end{lemma}

\begin{theorem}\label{topspacehyper}
	Let $X$ be a compact metric space with a metric $\rho$. Then $dim_L(X) \leq k-1$ if and only if for any $\epsilon > 0$ there exists a number $0 < \delta < \epsilon$ and a hypergraph $\mathcal{H}(\epsilon)$ on a finite vertex set $V(\mathcal{H}(\epsilon))$ with an edge set $E(\mathcal{H}(\epsilon)) = \{e_x : x\in X\}$,  which satisfies the following conditions: 
	
	1) $rank(\mathcal{H}(\epsilon)) \leq k$;
	
	2) $e_x\cap e_y \ne \emptyset$ for every $x,y\in X$ such that $\rho(x,y) < \delta$;
	
	3) $\rho(x,y) < \epsilon$ for every $x,y\in X$ such that $e_x \cap e_y \ne \emptyset$;
	
	4) for every $v\in V(\mathcal{H}(\epsilon))$ the set $X_v = \{x\in X : v\in e_x \}$ is open.
\end{theorem}

\begin{proof}

Suppose that $dim_L(X) \leq k$, $\epsilon > 0$ and let $\mathcal{C}$ be the corresponding $(\epsilon,k)$-cover of $X$. Since $X$ is compact, we can assume that a cover $\mathcal{C}$ is finite, i.e. $\mathcal{C} = \{C_1,...,C_m\}$. Let $\delta$ be the Lebesgue number of $\mathcal{C}$.

For a point $x\in X$ let $e_x = \{i\in [m] : x\in C_i\}$. Consider a hypergraph $\mathcal{H}$ with $V(\mathcal{H}) = [m]$ and $E(\mathcal{H}) = \{e_x : x\in X\}$. Then $\mathcal{H}$ satisfies conditions 1)-4). Indeed, $rank(\mathcal{H}) \leq k$, since $\mathcal{C}$ is $k$-cover. If $\rho(x,y) < \delta$, then by Lemma \ref{lebnumber} there is $i\in [m]$ such that $\{x,y\}\in C_i$, i.e. $i\in e_x\cap e_y$. Condition $j\in e_x\cap e_y$ means that $x,y\in C_j$,  and so $\rho(x,y) < \epsilon$, since $diam(C_j) < \epsilon$. Finally, for every $i\in V(\mathcal{H})$ we have $X_v = C_v$, and thus $X_v$ is open.

Conversely, let $\mathcal{H}$ be a hypergraph with $V(\mathcal{H}) = [m]$ satisfying conditions (1)-(4). Then it is straightforward to check that $\mathcal{C} = \{X_1,...,X_m\}$ is an open $(\epsilon,k)$-cover of $X$.
\end{proof}

 So, $dim_L(X)\leq k$ whenever for any $\epsilon > 0$ there is a hypergraph $\mathcal{H}(\epsilon)$ of $rank(\mathcal{H}(\epsilon)) \leq k$ with edges in bijective correspondence with points of $X$ such that two points are close if and only if corresponding edges intersect. 

Clique cover consisting of all edges of $G$ is a $\Delta(G)$-cover. It implies the following upper bound for $dim_L(G)$:

\begin{proposition}\label{prop:dimLDelta}
$dim_L(G)  \leq \Delta(G) - 1$. The equality holds, if $G$ is triangle-free.
\end{proposition}

\section{Hausdorff dimension of graphs}\label{measuregraph}

The goal of this section is to demonstrate that the complement product dimension is a graph-theoretical analogue of the Hausdorff dimension. First, we establish a formal connection by proving that the complement product dimension is associated with a graph measure analogous to the Hausdorff measure of topological spaces. Second, we demonstrate how this dimension is related to graph self-similarity.

\subsubsection*{Graph Measure}
In order to rigorously define a graph analogue of Hausdorff dimension, we need to define the corresponding measure first. Note that in any meaningful finite graph topology every set is a Borel set. As mentioned above, for measurable Borel sets in $\mathbb{R}^n$ Jordan, Lebesgue and Hausdorff measures are equivalent.  Thus further we will consider the graph analogue of Jordan measure. We propose a parameter which is aimed to serve as the graph analogue of the Jordan measure, and prove that it indeed satisfies the axioms of measure. Finally, based on this parameter we define the Hausdorff dimension of a graph. 

It is known, that every graph $G$ is isomorphic to an induced subgraph $G'$ of a categorical product $K^1_{n_1}\times\dots \times K^d_{n_d}$ of complete graphs \cite{hell2004graphs}.  Without loss of generality we may assume that $n_1 = ... = n_d = n$, i.e. $G'$ is an induced subgraph of the graph $\mathbb{S}_d = (K_n)^d$. $\mathbb{S}_d$ will be referred to as {\it a space} of dimension $d$ and $G'$ as an {\it embedding} of $G$ into $\mathbb{S}_d$. After assuming that $V(K_n) = [n]$, we may say that every vertex $v\in \mathbb{S}_d$ is a vector $v = (v_1,...,v_d)\in [n]^d$, and two vertices $u$ and $v$ are adjacent if and only if $v_r\ne u_r$ for every $r\in [d]$.

{\it Hyper-rectangle} $R = R(J_1,...,J_d)$ is a subgraph of $\mathbb{S}_d$, that is defined as follows:  $R = K_n[J_1]\times\dots\times K_n[J_d]$, where for every $i\in [d]$ the set $J_i \subseteq [n]$ is non-empty. The {\it volume} of $R$ is naturally defined as $vol(R) = |V(R)| = \prod_{i=1}^d |J_i|$.

The family $\mathcal{R} = \{R^1,...,R^m\}$ of hyper-rectangles is a {\it rectangle co-cover} of $G'$, if the subgraphs $R^i$ are pairwise vertex-disjoint, $V(G')\subseteq \bigcup_{i=1}^m V(R^i)$ and $\mathcal{R}$ covers all non-edges of $G'$, i.e. for every pair of non-adjacent vertices $x,y\in V(G')$ there exists $j\in[m]$ such that $x,y\in V(R^j)$. We define {\it $d$- volume} of a graph $G$ as

\begin{equation}\label{grvol}
vol^d(G) = \min_{G'}\min_{\mathcal{R}} \sum_{R\in \mathcal{R}} vol(R),
\end{equation}

\noindent
where the first minimum is taken over all embeddings $G'$ of $G$ into $d$-dimensional spaces $\mathbb{S}_d$ and the second minimum - over all rectangle co-covers of $G'$. For example, Fig. \ref{fig:selfSim} (left) demonstrates that the 2-dimensional volume of the path $P_4$ is equal to 6. 

% \begin{figure}[h]
% \begin{center}
% \includegraphics[width=60mm,height=60mm,keepaspectratio] {EmbedP4_1.pdf}
% \caption{\label{fig:embedP4} {\small Embedding of $G=P_4$ into a 2-dimensional space $S=(K_3)^2$ and its rectangle co-cover by a hyper-rectangle of volume $6$}}
% \end{center}
% \end{figure}

 \begin{figure}%
    \centering
    \subfloat{{\includegraphics[width=4cm]{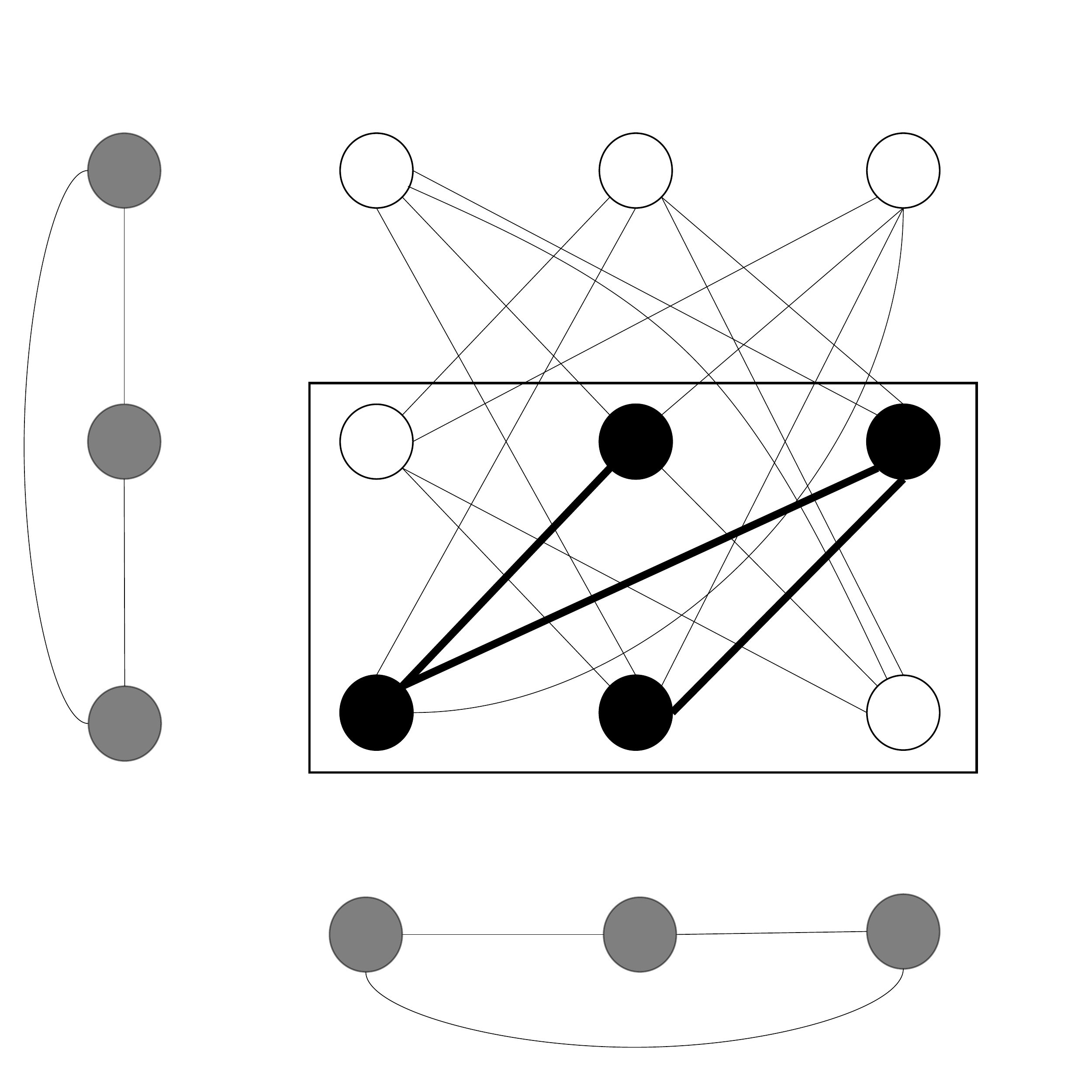} }}%
    \qquad
    \subfloat{{\includegraphics[width=10cm]{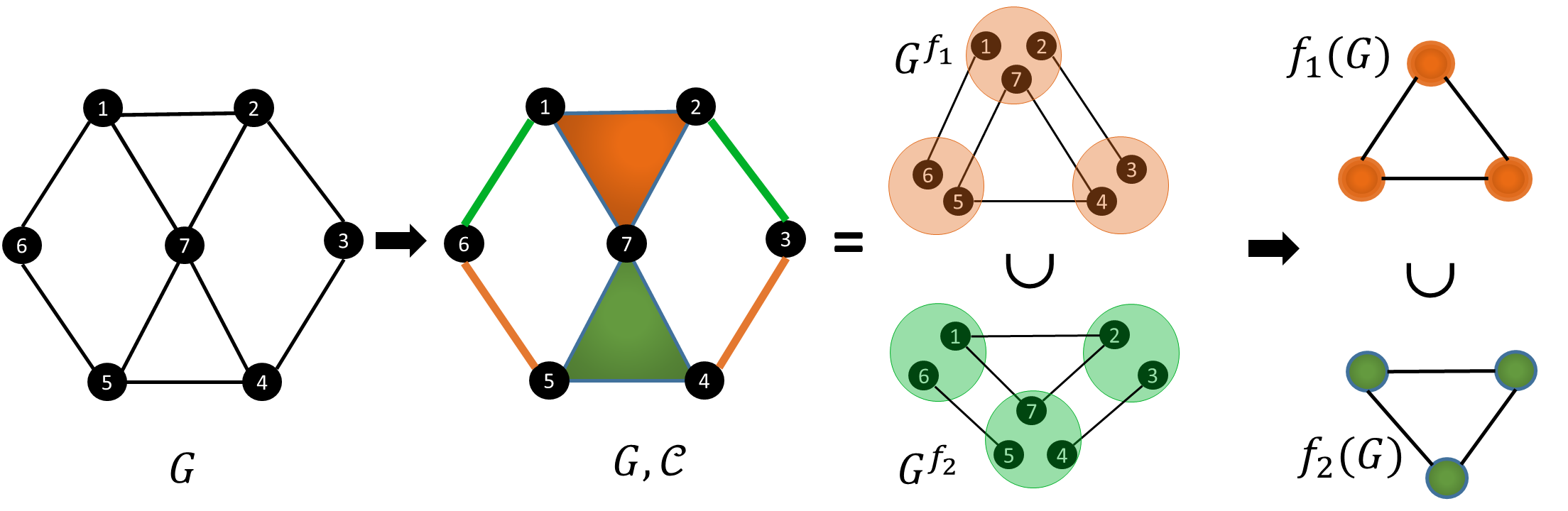} }}%
    \caption{\footnotesize {\bf Left}: Embedding of $G=P_4$ into a 2-dimensional space $\mathbb{S}_2=(K_3)^2$ and its rectangle co-cover by a hyper-rectangle of volume $6$. {\bf Right}: equivalent $k$-cover defining a self-similarity of a graph $G$. From left to right: the graph $G$; an equivalent 2-cover $\mathcal{C}$ of $G$ with the clusters of the same color highlighted in red and green; subgraphs $G^{f_1}$ and $G^{f_2}$ such that $G = G^{f_1}\cup G^{f_2}$ for the contracting family $(f_1,f_2)$ defined by $\mathcal{C}$; contractions $f_1(G)$ and $f_2(G)$}%
    \label{fig:selfSim}%
\end{figure}

Based on the definition of $d$-volume, we define a {\it $d$-measure} of a graph $F$ as follows:

\begin{equation}\label{neasuregraph}
\mathcal{H}^d(F) = \left \{
\begin{array}{lll}
vol^d(\overline{F}), \text{ if } \overline{F} \text{ has an embedding into } \mathbb{S}_d;
\\ +\infty, \text{ otherwise }
\end{array}
\right.
\end{equation}

The main theorem of this section confirms that $\mathcal{H}^d$ indeed satisfy the axioms of a measure: 

\begin{theorem}\label{measureadditivity}
Let $F_1$ and $F_2$ be two graphs, and $F^1\cup F^2$ is their disjoint union. Then 

\begin{equation}\label{eq:measaddit}
\mathcal{H}^d(F^1\cup F^2) = \mathcal{H}^d(F^1) + \mathcal{H}^d(F^2)
\end{equation}

\end{theorem}

\begin{proof}
It can be shown \cite{babai1992linear}, that $\overline{F^1\cup F^2}$ can be embedded into a categorical product of $d$ complete graphs if and only if both $\overline{F_1}$ and $\overline{F_2}$ have such embeddings. Therefore the relation (\ref{eq:measaddit}) holds, if some of the terms are equal to $+\infty$. If all $\overline{F_1}$,$\overline{F_2}$, $\overline{F^1\cup F^2}$ can be embedded into $\mathbb{S}_d$, we proceed with the series of claims.

Let $W^1,W^2\subseteq V(\mathbb{S}_d)$. We write $W_1\sim W_2$, if every vertex from $W_1$ is adjacent to every vertex from $W_2$. Denote by $P_k(W_1)$ {\it the $k$-th projection} of $W_1$, i.e. the set of all $k$-coordinates of vertices of $W_1$: $P_k(W_1) = \{v_k : v\in W_1\}.$
In particular, $P_k(R(J_1,...,J_d)) = J_k$. The following claim follows directly from the definition of $\mathbb{S}_d$.

\begin{claim}\label{padjcoord}
$W^1 \sim W^2$ if and only if $P_k(W_1)\cap P_k(W_2) = \emptyset$ for every $k\in [d]$.
\end{claim}

Assume that $\mathcal{R} = \{R^1,...,R^m\}$ is a minimal rectangle co-cover of a minimal embedding $G'$, i.e.  $vol^d(G) = \sum_{R\in \mathcal{R}} vol(R)$.  We will demonstrate, that $\mathcal{R}$ has a rather simple structure. 

\begin{claim}\label{dvolumeinterindex}
Let $J_k^i = P_k(R^i)$. Then $J^i_k \cap J^j_k = \emptyset$ for every $i,j\in [m]$, $i\ne j$ and every $k\in [d]$ .
\end{claim}

\begin{proof}
First note, that for every hyper-rectangle $R^i = R(J^i_1,...,J^i_d)$, every coordinate $k\in [d]$ and every $l\in J^i_k$ there exists $v\in V(G')\cap V(R^i)$ such that $v_k = l$. Indeed, suppose that it does not hold for some $l\in J^i_k$. If $|J_k^i| = 1$, then it means that $V(G')\cap V(R^i) = \emptyset$. Thus, $\mathcal{R}' = \mathcal{R}\setminus \{R^i\}$ is a rectangle co-cover, which contradicts the minimality of $\mathcal{R}$. If $|J_k^i| > 1$,  consider a hyper-rectangle  $(R^i)' = R(J_i^1,...,J_k^i\setminus \{l\},...,J^i_d)$. The set $\mathcal{R}' = \mathcal{R}\setminus \{R^i\}\cup \{(R^i)'\}$  is a rectangle co-cover, and the $d$-volume of $(R^i)'$ is smaller than the $d$-volume of $R^i$. Again it contradicts minimality of $\mathcal{R}$.

Now assume that for some distinct $i,j\in [m]$ and $k\in [d]$ we have $J^i_k \cap J^j_k \supseteq \{l\}$. Then there exist $u\in V(G')\cap V(R^i)$ and $v\in V(G')\cap V(R^j)$ such that $u_k = v_k = l$. So, $uv\not\in E(G')$, and therefore by the definition $uv$ is covered by some $R^h\in \mathcal{R}$. The hyper-rectangle $R^h$ intersects both $R^i$ and $R^j$, which contradicts the definition of a rectangle co-cover.
\end{proof}

\begin{claim}\label{dvolumecoconcomp}
Let $U^i = V(G')\cap V(R^i)$, $i=1,...,m$. Then the set $U=\{U^1,...,U^m\}$ coincides with the set of co-connected components of $G'$.
\end{claim}
\begin{proof}
Claims \ref{padjcoord} and \ref{dvolumeinterindex} imply, that vertices of distinct hyper-rectangles from the co-cover $\mathcal{R}$ are pairwise adjacent. So, $U^i \sim U^j$ for every $i,j\in [m]$, $i\ne j$.

Let $\mathcal{C}=\{C^1,...,C^r\}$ be the set of co-connected components of $G'$ (thus $V(G') = \bigsqcup_{l=1}^r C^l = \bigsqcup_{i=1}^m U^i$). Consider a component $C^l\in \mathcal{C}$ and the sets $C^l_i = C^l \cap U^i$, $i=1,...,m$. We have $C^l = \bigsqcup_{i=1}^m C^l_i$ and $C^l_i \sim C^l_j$ for all $i\ne j$. Therefore, due to co-connectedness of $C^l$, exactly one of the sets $C^l_i$ is non-empty.

So, we have demonstrated, that every co-connected component $C^l$ is contained in some of the sets $U^i$. Now, let some $U_i$ consists of several components, i.e. without loss of generality $U_i = C^1\sqcup\dots\sqcup C^q$, $q\geq 2$. Let $I_k^j = P_k(C^j)$, $j=1,...,q$. By Proposition \ref{padjcoord} we have $I_k^{j_1} \cap I_k^{j_2} = \emptyset$ for all $j_1\ne j_2$, $k=1,...,d$.  Consider hyper-rectangles $R^{i,1} = R(I_1^1,...,I_d^1)$,...,$R^{i,q} = R(I_1^q,...,I_d^q)$. Those hyper-rectangles are pairwise vertex-disjoint, and $C^j \subseteq V(R^{i,j})$ for all $j\in [q]$. Since every pair of non-connected vertices of $G'$ is contained in some of its co-connected components, we arrived to the conclusion, that the set $\mathcal{R}' = \mathcal{R}\setminus \{R^i\} \cup \{R^{i,1},...,R^{i,q}\}$ is a rectangle co-cover. Moreover, $V(R^{i,1})\sqcup\dots\sqcup V(R^{i,q}) \subsetneq V(R^i)$, and therefore $\sum_{j=1}^q vol(R^{i,j}) < vol(R^i)$, which contradicts the minimality of $\mathcal{R}$.
\end{proof}

\begin{claim}\label{voljoin}
Let $\mathcal{C}=\{C^1,...,C^m\}$ be the set of co-connected components of $G'$. Then $R_i = R(P_1(C^i),...,P_d(C^i))\supseteq C^i$.
\end{claim}

\begin{proof}
By Claim \ref{dvolumecoconcomp}, $|\mathcal{R}| = |\mathcal{C}|$, and every component $C^i\in \mathcal{C}$ is contained in a unique hyper-rectangle $R^i\in \mathcal{R}$, $i=1,...,m$. Every pair of non-adjacent vertices of $G'$ belong to some of its co-connected components. This fact, together with the minimality of $\mathcal{R}$, implies that $R^i$ is the minimal hyper-rectangle that contains $C^i$. Thus $R_i = R(P_1(C^i),...,P_d(C^i))$.
\end{proof}

\begin{claim}\label{cor:minvolspace}
If $\overline{G}$ is connected, than $vol^d(G)$ is the minimal volume of all $d$-dimensional spaces $S$, where $G$ can be embedded. 
\end{claim}

\begin{claim}\label{volconcomp}
Let $\mathcal{D}=\{D^1,...,D^m\}$ be the set of co-connected components of $G$. Then $vol^d(G) = \sum_{i=1}^m vol^d(G[D^i])$.
\end{claim}
\begin{proof}
Suppose that $\{C^1,...,C^m\}$ is the set of co-connected components of $G'$, and $G[D^i]\cong G'[C^i]$. Claims \ref{dvolumecoconcomp} and \ref{voljoin} imply that $\{R_i\}$ is a rectangle co-cover of an embedding of $G[D^i]$.  Therefore we have $vol(R^i) \geq vol^d(G[D^i])$ and thus $vol^d(G) = \sum_{i=1}^m vol(R^i) \geq \sum_{i=1}^m vol^d(G[D^i])$. So, it remains to prove the inverse inequality.

Let $G'^i$ be a minimal embedding of $G[D^i]$ into $\mathbb{S}_d$.  By Claim \ref{dvolumecoconcomp}, every minimal hyper-rectangle co-cover of $G'^i$ consists of a single hyper-rectangle $R^i = R(J^i_1,...,J^i_d)$ or, in other words, $G[D^i]$ is embedded into $R^i$. Now we can construct an embedding $G'$ of $G$ into $S=(K_{nm})^d$ and its hyper-rectangle co-cover. It can be done as follows. Consider $I^i_k =  \{(i-1)m + l : l\in J^i_k\}$, $k=1,...,d$, and let $Q^i = K_{mn}(I^i_1)\times\dots\times K_{mn}(I^i_d)$. Obviously, $Q^i \cong R^i$. Now embed $G[D^i]$ into $Q^i$. Let $G'^i$ be those embeddings. By Claim \ref{padjcoord}, $V(G'^i) \sim V(G'^j)$, so $G' = G'^1\cup...\cup G'^m$ is indeed an embedding of $G$.

All hyper-rectangles $Q^i$ are pairwise disjoint. Since every pair of non-adjacent vertices belong to some subgraph $G[D^i]$, we have that $\mathcal{Q} = \{Q^1,...,Q^m\}$ is hyper-rectangle co-cover of $G'$. Therefore $vol^d(G) \leq \sum_{i=1}^m vol(Q^i) = \sum_{i=1}^m vol^d(G[D^i])$.
\end{proof}

Now the equality (\ref{eq:measaddit}) follows from Corollary \ref{volconcomp}. In concludes the proof of Theorem.
\end{proof}

Following the analogy with Hausdorff dimension of topological spaces (\ref{hdimtop}), we define a {\it Hausdorff dimension} of a graph $G$ as 

\begin{equation}\label{hdimgraph}
dim_H(G) = \min\{s \geq 0 :\mathcal{H}^s(G) < \infty\}-1.
\end{equation}

\noindent
Thus, Hausdorff dimension of a graph can be identified with a Prague dimension of its complement. 
% Note that both Lebesgue and Hausdorff dimensions are monotone with respect to induced subgraphs.

According to Theorem \ref{thm:pdimcharact}, graph Hausdorff dimension is defined by the existence of a separating equivalent clique cover. In a typical case the coloring requirement is more important than separation requirement. Indeed, two vertices may not be separated by some cluster of a given clique cover only if these two vertices are {\it true twins}, i.e. they have the same closed neighborhoods. In most network models and experimental networks presence of such vertices in highly unlikely; besides in most situations they can be collapsed into a single vertex without changing the majority of important network topological properties. 
% Graphs without true twins will be further called {\it irreducible graphs}, and we will concentrate on them unless stated otherwise.

\subsubsection*{Self-similarity}

The self-similarity of compact metric space $(X,d)$ is defined using the notion of a {\it contraction} \cite{edgar2007measure}. An open mapping $f:X\rightarrow X$ is a {\it similarity mapping}, if $d(f(u),f(v)) \leq \alpha d(u,v)$ for all $u,v\in X$, where $\alpha$ is called its {\it similarity ratio} (such mapping is obviously continuous). If $\alpha < 1$, then it is a {\it contraction}. The space $X$ is {\it self-similar}, if there exists a family of contractions $f_1,...,f_k$ such that $X = \bigcup_{i=1}^k f_i(X)$.

This definition cannot be directly applied to discrete metric spaces such as graphs, since for them contractions in the strict sense do not exist. To formally and rigorously define the self-similarity of graphs, we proceed as follows. It is convenient to assume that every vertex is adjacent to itself. For two graphs $G$ and $H$, a {\it homomorphism} \cite{hell2004graphs} is a mapping $f:V(G)\rightarrow V(H)$ which maps adjacent vertices to adjacent vertices, i.e. $f(u)f(v)\in E(H)$ for every $uv\in E(G)$. A homomorphism $f$ is a {\it similarity mapping}, if inverse images of adjacent vertices are also adjacent, i.e. $uv\in E(G)$ whenever $f(u)f(v)\in E(H)$ (it is possible that $f(u)=f(v)$). In other words, for a similarity mapping, images and inverse images of cliques are cliques. With a similarity mapping $f$ we can associate a subgraph $G^f$ of $G$, which is formed by all edges $uv$ such that $f(u)\ne f(v)$ (Fig. \ref{fig:selfSim}).

A family of graph similarity mappings $f_i:V(G) \rightarrow V(G_i)$, $i=1,...,k$, is a {\it contracting family}, if every edge of $G$ is contracted by some mapping, i.e. for every $uv\in E(G)$ there exists $i\in\{1,...,k\}$ such that $f_i(u) = f_i(v)$. The graphs $G_i$ are {\it contractions} of $G$. Finally, a graph $G$ is {\it self-similar}, if $G = \bigcup_{i=1}^k G^{f_i}$ (Fig. \ref{fig:selfSim}).

\begin{proposition}\label{prop:contrdim}
Graph $G$ is self-similar with a contracting family $f_1,...,f_k$ if and only if there is an equivalent separating $k$-cover of $G$.
\end{proposition}

\begin{proof}
For a given contracting family $f_1,...,f_k$ and any $i\in \{1,...,k\}$, the sets $\mathcal{C}_i = \{f_i^{-1}(v): v \in V(G_i)\}$,  consist of disjoint cliques. By the definition, every edge of $G$ is covered by one of these cliques.  Therefore $\mathcal{C} = \mathcal{C}_1\cup...\cup \mathcal{C}_k$ is an equivalent $k$-cover of $G$. Furthermore, due to the self-similarity of $G$, for every edge $xy\in E(G)$ there is a mapping $f_i$ that does not contract it, i.e. $f_i(x) = u\neq v = f_i(y)$. Thus, $x$ and $y$ are separated by the cliques $f_i^{-1}(u)$ and $f_i^{-1}(v)$, and therefore $\mathcal{C}$ is a separating cover. 

Conversely, let $\mathcal{C} = \mathcal{C}_1\cup...\cup \mathcal{C}_k$ be a separating equivalent $k$-cover, where $\mathcal{C}_i = (C_i^1,...,C_i^{r_i})$ is the set of connected components of the $i$th equivalence graph (some of them may consist of a single vertex). Construct a graph $G_i$ by contracting every clique $C_i^j$ into a single vertex $v_i^j$ and the mapping $f_i$ by setting $f_i(C_i^j) = v_i^j$. Then the collection $(f_1,...,f_k)$ is a contracting family.
\end{proof}

According to Proposition \ref{prop:contrdim}, all graphs could be considered as self-similar  - for example, we can construct $|E(G)|$ trivial similarity mappings by individually contracting each edge. Thus, it is natural to concentrate our attention on non-trivial similarity mappings and measure the degree of the graph self-similarity by the minimal number of similarity mappings in a contracting family, i.e. by its Hausdorff dimension. Smaller number of similarity mappings indicates the denser packing of a graph by its contraction subgraphs, i.e. the higher self-similarity degree. In particular, the {\it normalized Hausdorff dimension} $\overline{dim}_H(G) = dim_H(G)/|V(G)|$ could serve as a measure of self-similarity. 

% \begin{figure}[h]
% \begin{center}
% \includegraphics[width=0.75\textwidth,height=\textheight,keepaspectratio] {selfSim.png}
% \caption{\label{fig:selfSim} {\footnotesize Transformations of a graph $G$ (left). }}
% \end{center}
% \end{figure}

\section{Fractal graphs: theoretical study}\label{sec:fractheory}

In this section, we consider only connected graphs. Importantly, the relation (\ref{lebvshaus}) between Lebesgue and Hausdorff dimensions of topological spaces remains true for graphs.

\begin{proposition}\label{lebvshausgraph}
For any graph $G$, $dim_R(G) - 1 = dim_L(G) \leq dim_H(G) = dim_P(\overline{G})-1.$
\end{proposition}

\begin{proof}
Let product dimension of a graph $\overline{G}$ is equal to $k$. Then by Theorem \ref{thm:pdimcharact} $G$ is an intersection graph of strongly $k$-colorable hypergraph. Since rank of every such hypergraph  obviously does not exceed $k$, Theorem \ref{thm:covervsinter} implies, that $dim_R(G)\leq k$.
\end{proof}

Proposition \ref{lebvshausgraph} allows us to define  graph fractals analogously to the definition of fractals for topological spaces: a graph $G$ is a {\it fractal}, if $dim_L(G) < dim_H(G)$, i.e. $dim_R(G) < dim_P(\overline{G})$. In particular, we say that a fractal graph $G$ is {\it $k$-fractal}, if  $dim_L(G) = k$. For example, the graph $G$ on Fig. \ref{fig:selfSim} (right) is self-similar, but not fractal, since $dim_L(G) = dim_H(G) = 1$. In contrast, Fig. \ref{fig:serpGasket} demonstrates that Sierpinski gasket graph $S_3$ is 1-fractal (these graphs are studied in detail in the following section). 
The only connected 0-fractals are complete graphs $K_n$. Next, we study fractal graphs of higher dimensions. 
% We will concentrate on $K_4$-free and $K_3$-free graphs. Consideration of such graphs is well-justified from the point of view of both graph theory, where many classes of such graphs have been extensively studied, and network theory, where important network models have this property. Examples of such models will be considered below.  

% It can be observed, that in general the difference between Lebesgue and Hausdorff dimensions of a graph can be arbitrarily large. To see it, we may employ the known fact \cite{hamelink1968partial} that for any triangle-free graphs $H$ there exists a graph $G(H)$ such that vertices of $H$ correspond to maximal cliques of $G$, and two vertices of $H$ are adjacent whenever the corresponding cliques intersect. Let $H=H_k$ be a Mycielskian \cite{mycielski1955coloriage}, i.e. a specially constructed triangle-free graph, whose chromatic number is equal to the predefined integer $k$. In the graph $G(H)$, every vertex belongs to at most two maximal cliques (otherwise $H$ will have triangles), and   

\subsubsection*{Triangle-free graphs}  Let $\chi'(G)$ denotes the edge chromatic number of a graph $G$. Classical Vizing's theorem \cite{vizing1964estimate} states that $\Delta(G) \leq \chi'(G) \leq \Delta(G) + 1$, i.e. the set of all graphs can be partitioned into two classes: graphs, for which $\chi'(G) = \Delta(G)$ (class 1) and graphs, for which $\chi'(G) = \Delta(G) + 1$ (class 2).

By Proposition \ref{prop:dimLDelta}, $dim_L(G) = dim_R(G) - 1 = \Delta(G)-1$, if $G$ contains no triangles. For such graphs we have 

\begin{proposition}\label{prop:class2}
Triangle-free fractals are exactly triangle-free graphs of class 2. 
\end{proposition}

\begin{proof}
    The statement holds, if  $G=K_2$. Suppose that $G$ has $n\geq 3$ vertices. For such graphs, every clique cover is a collection of its edges and vertices. However, since $G$ is connected, for every pair of vertices there is an edge that separates them. Therefore we may assume that the clique cover consists only of edges, and a feasible assignment of colors to the cliques is an edge coloring. Thus,  it is true that $dim_P(\overline{G}) = \chi'(G)$ (i.e. $dim_H(G) = \chi'(G)-1$), and the statement of the proposition follows. 
\end{proof}

In particular, bipartite graphs are triangle-free graphs of class 1 \cite{konig1916graphen}. Therefore bipartite graphs (and trees in particular) are not fractals, even though some of them may have high degree of self-similarity (e.g. binary trees). It also should be noted that although some known geometric fractals are called trees (e.g. so-called $H$-trees), they are not discrete object, and their fractality is associated with their drawings on a plane; thus our framework does not apply to them. 

% Furthermore, typical Erd\"{o}s-Renyi graph with high probability has a unique vertex of the maximum degree. For such graphs, it is also known that $\Delta(G) = \chi'(G)$ \cite{vizing1965critical}. Therefore with high probability sparse  Erd\"{o}s-Renyi random graphs are also not fractal.

 \subsubsection*{Scale-free graphs} Recall that {\it scale-free networks} are graphs whose degree distribution (asymptotically) follows the power-law, i.e. the probability that a given vertex has a degree $d$ could be approximated by the function $ad^{-\alpha}$, where $a$ is a constant and $\alpha$ is a {\it scaling exponent}. There is a number of models of scale-free networks of different degree of mathematical rigour known in the literature, including various modifications of the preferential attachment scheme. Following \cite{janson2010large,chung2002connected}, we will consider a more formal probabilistic model. Assume without loss of generality that $a = 1$ \cite{janson2010large}. For each vertex $i\in \{1,...,n\}$, we assign a weight $w_i = (\frac{n}{i})^{1/\alpha}$. Then we construct a graph $G(n,\alpha)$ by independently connecting any pair of vertices $i,j$ by an edge with the probability $p_{ij} = 1 - e^{-\lambda_{ij}}$, where $\lambda_{ij} = b\frac{w_iw_j}{n}$ and $b$ is a constant.
 
 From now on we will use the following standard nomenclature \cite{brandstadt1999graph}. An induced subgraph isomorphic to a cycle is {\it a hole}, a hole with the odd number of vertices is {\it an odd hole}. The star $K_{1,3}$ is {\it the claw}, the 4-vertex graph consisting of two triangles with a common edge is {\it the diamond} and the 5-vertex graph consisting of two triangles with a common vertex is {\it the butterfly}. 
 
\begin{theorem}\label{thm:dimsf}
For graphs $G = G(n,\alpha)$ with $\alpha > \frac{12}{5}$ and $\Delta = \Delta(G)$, with high probability $\dim_L(G) \in \{\Delta-2,\Delta-1\}$ and $\dim_H(G) \in \{\Delta-2,\Delta-1,\Delta,\Delta+1\}$.
\end{theorem}

\begin{proof}

 It has been proved in \cite{janson2010large} that the clique number of a graph $G(n,\alpha)$  with the scaling constant $\alpha > 2$ is either 2 or 3 with high probability, i.e. $p(\omega(G(n,\alpha)) \in \{2,3\}) \rightarrow 1$ as $n\rightarrow \infty$ for $n$-vertex scale-free graphs $G(n,\alpha)$ that have power-law degree distribution with the exponent $\alpha$. The following lemma complements this fact: 
 
 \begin{lemma}\label{lem:nodiamonds}
1) For $\alpha > \frac{12}{5}$, graphs $G(n,\alpha)$ with high probability do not contain diamonds. 

2) For $\alpha > 2$, graphs $G(n,\alpha)$ with high probability do not contain butterflies. 
 \end{lemma}
 
 \begin{proof}
1) Let vertices $i,j,k,l$ form a diamond, where $i$ and $j$ are non-adjacent. For the probability of this event, we have 
 
 $$p' = p_{ik}p_{jk}p_{il}p_{jl}p_{kl}(1-p_{ij})\leq p_{ik}p_{jk}p_{il}p_{jl}p_{kl}\leq \lambda_{ik}\lambda_{jk}\lambda_{il}\lambda_{jl}\lambda_{kl} = \frac{b^5}{n^5}w_i^2w_j^2w_k^3w_l^3$$
 
 Thus, the total probability that these vertices form a diamond can be estimated as
 
 \begin{align}\label{eq:probdiamond}
 p \leq \frac{b^5}{n^5} w_i^2w_j^2w_k^2w_l^2(w_iw_j + w_iw_k + w_iw_l + w_jw_k + w_jw_l + w_kw_l) \leq 6 \frac{b^5}{n^5} w_i^3w_j^3w_k^3w_l^3.
 \end{align}
 
 Let $X_D$ be the number of diamonds in $G(n,\alpha)$. For the expected value of this random variable we have 
 \begin{align}\label{eq:expdiamond}
     E(X_D) \leq 6 \frac{b^5}{n^5} \sum_{\{i,j,k,l\}\subseteq \{1,...,n\}} w_i^3w_j^3w_k^3w_l^3 \leq 6 \frac{b^5}{n^5} \left(\sum_{i=1}^n w_i^3\right)^4 =  6 \frac{b^5}{n^5}\left(\sum_{i=1}^n \left(\frac{n}{i}\right)^{\frac{3}{\alpha}}\right)^4
 \end{align}
 
 Using an integral upper bound, it is easy to see that 
 
 \begin{equation}\label{eq:ubsum}
     g_r(n) = \sum_{i=1}^n \left(\frac{n}{i}\right)^r = 
     \left \{
\begin{array}{lll}
O(n^r), \text{ if } r > 1; 
\\ O(n\log(n)), \text{ if } r=1
\\ O(n), \text{ if } r < 1.
\end{array}
\right.
 \end{equation}
 \noindent
 Furthermore, $g_r(n) \leq g_s(n)$ whenever $r\leq s$.
 
 Select $\beta \in (0,\frac{1}{4})$ such that $\alpha \geq p = \frac{12}{5}(1 + \beta)$. Then we have $g_{\frac{3}{\alpha}}(n) \leq g_{\frac{3}{p}}(n) = O(n^{\frac{3}{p}})$. Thus $E(X_D) =  O(n^{\frac{12}{p}-5}) = O(n^{-5(1-\frac{1}{1+\beta})}) = o(1)$. Finally, by Markov's inequality we have $p(X_D \geq 1) \leq E(X_D) \rightarrow 0$ as $n \rightarrow +\infty$.

 2) Similarly to 1), for the probability $p'$ that vertices $i,j,k,l,r$ form a butterfly with the center $i$ we have 
 
 $$p' \leq \frac{b^6}{n^6}w_i^4w_j^2w_k^2w_l^2w_r^2.$$
 
 Thus, for the number of butterflies $X_B$ its expectation satisfy the following chain of inequalities:
 
  \begin{align}\label{eq:expbutterfly}
     E(X_D) \leq \frac{b^6}{n^6} \sum_{i=1}^n w_i^4 \sum_{\{j,k,l,r\}\subseteq \{1,...,n\}} w_j^2w_k^2w_l^2w_r^2 \leq  \frac{b^6}{n^6} \sum_{i=1}^n w_i^4 \left(\sum_{i=1}^n w_i^2\right)^4 \\ =  \frac{b^6}{n^6} \sum_{i=1}^n \left(\frac{n}{i}\right)^{\frac{4}{\alpha}} \left(\sum_{i=1}^n\left(\frac{n}{i}\right)^{\frac{2}{\alpha}}\right)^4 = \frac{b^6}{n^6}g_{\frac{4}{\alpha}}(n)\left(g_{\frac{2}{\alpha}}(n)\right)^4
 \end{align}
 
 Like in 1), select a small number $\beta > 0$ such that $\alpha \geq p = 2(1 + \beta$). Then by (\ref{eq:ubsum}) we have $g_{\frac{4}{\alpha}}(n) \leq g_{\frac{4}{p}}(n) = O(n^{\frac{4}{p}})$ and $g_{\frac{2}{\alpha}}(n) = O(n)$. Thus $E(X_D) =  O(n^{4 + \frac{4}{p}-6}) = O(n^{-2(1-\frac{1}{1 + \beta})}) = o(1)$. After applying Markov's inequality the claim follows. 
  \end{proof}
  
  \vspace{-2ex}
  Thus, a typical scale-free graph with the scaling exponent $\alpha > \frac{12}{5}$ has only cliques of size 2 and 3, and every vertex belongs to at most one triangle. For such graphs, a minimal clique cover defining $\dim_L(G)$ consists of all triangles and the edges that do not belong to triangles. Let $tr(v)\in \{0,1\}$ be the number of triangles which include a vertex $v$. Then with high probability $\dim_L(G) = \max_v (\deg(v) - tr(v)) - 1$, and the statement about $\dim_L(G)$ follows. By Vizing's theorem, two-vertex clusters of this cover could be colored by at most $\Delta + 1$ color, and one additional color could be used to color the triangles. Thus, $\dim_L(G) \leq \dim_H(G) \leq \Delta + 1$, which proves the statement for $\dim_H(G)$.
  \end{proof}

Note that the similar arguments may be used to prove Theorem \ref{thm:dimsf} for more general model, when the weights $w_i$ are identically distributed random variables with power-law-distributed tail (see \cite{janson2010large}). In particular, for preferential attachment graph $G_t$ following this model, Theorem \ref{thm:dimsf} and the estimations of $\Delta(G_t)$ from \cite{flaxman2005high,bollobas2003mathematical}  imply that its Hausdorff dimension is roughly asymptotically equivalent to $\sqrt{t}$ (up to an arbitrarily slowly growing multiplicative factor).

% \begin{claim}
% For graphs $G = G(n,\alpha)$ with $\alpha > \frac{12}{5}$, with high probability $\dim_L(G) \in \{\Delta(G)-1, \Delta(G)\}$, where the exact value of $\dim_L(G)$ depends on whether all max-degree vertices belong to triangles.
% \end{claim}

\subsubsection*{Erd\"{o}s-Renyi graphs}
Similar properties of dimensions  hold for sparse Erd\"{o}s-Renyi graphs $G(n,p)$, where $p = \frac{1}{n^{\alpha}}$:

\begin{theorem}\label{thm:dimER}
For Erd\"{o}s-Renyi graphs $G = G(n,p)$ with $\alpha > \frac{5}{6}$ and $\Delta = \Delta(G)$, with high probability $\dim_L(G) \in \{\Delta-2,\Delta-1\}$ and $\dim_H(G) \in \{\Delta-2,\Delta-1,\Delta,\Delta+1\}$.
\end{theorem}

Indeed, it is implied by the following simple statement and considerations analogous to the ones in Theorem \ref{thm:dimsf}:

 \begin{lemma}\label{lem:nodiamonds}
1) For $\alpha > \frac{2}{3}$, graphs $G(n,p)$ with high probability do not contain $K_4$. 

2) For $\alpha > \frac{4}{5}$, graphs $G(n,p)$ with high probability do not contain diamonds. 

3) For $\alpha > \frac{5}{6}$, graphs $G(n,p)$ with high probability do not contain butterflies. 
 \end{lemma}
 
 \begin{proof}
 Let $X_D$ be the number of diamonds in $G(n,p)$. Then $E(X_D) = \binom{n}{4}p^5(1-p) = O(n^{4 - 5\alpha}) = o(1)$, and 2) follows from Markov's inequality. Other statements can be proved analogously.
 \end{proof}
 
\vspace{-2ex}
Note that  for sparse Erd\"{o}s-Renyi graphs, $\Delta(G(n,\frac{c}{n})) = O(log(n))$ \cite{bollobas2001random} and thus $dim_H(G(n,\frac{c}{n})) = O(log(n))$. For dense Erd\"{o}s-Renyi graphs, the asymptotics is described by Theorem \ref{thm:almostalldim} (see Section on information-theoretic connection).
 
\subsubsection*{Cubic and subcubic graphs} A graph is cubic, if all its vertices have the degree three. Cubic graphs arise naturally in graph theory, topology, physics and network theory \cite{woodhouse2016stochastic,mcdiarmid2017modularity}, and has been extensively studied.  Among cubic graphs, the class of so-called {\it snarks} is distinguished in graph theory (the most famous snark - Petersen graph - is shown on Fig. \ref{fig:petersen}). Snark is defined as a biconnected cubic graph of class 2  \cite{chladny2010factorisation}. Snark is non-trivial, if it is triangle-free \cite{chladny2010factorisation}.  Snarks constitute important class of graphs, which has been studied for more than a century and whose structural properties continue to puzzle researchers to this day \cite{chladny2010factorisation}. Discovery of new non-trivial snarks is a valuable scientific result\footnote{Thus the name of this graph class introduced by M. Gardner \cite{gardner1977mathematical}}, not unlike the discovery of new fractals, with many known snarks also possessing high degree of symmetry and being constructed by certain recursive procedures. According to Proposition \ref{prop:class2}, this analogy is well-justified, as non-trivial snarks are indeed fractals according to our definition. As shown by Theorem \ref{thm:fraccubic}, the inverse relation  also holds, as cubic fractals could be reduced to snarks.

The connection between fractality and class 2 graphs continues to hold for wider class of subcubic graphs (i.e. the graphs with $\Delta(G) \leq 3$). Indeed, consider a graph $G_{-3}$ obtained from $G$ by removal of edges of all its triangles. Given that $dim_L(G) \leq 2$, the following theorem describes subcubic fractals:

\begin{theorem}\label{thm:fractaldelta3}
Let $G$ be a subcubic connected graph with $n\geq 5$ vertices. 
\begin{itemize}
\item[(1)] $G$ is 1-fractal, if and only if it is claw-free, but contains the diamond or an odd hole.
\item[(2)] $G$ is 2-fractal, if and only if it contains the claw and $G_{-3}$ is of class 2.
\end{itemize}
\end{theorem}

\begin{proof}
% It is easy to see by definition, that $dim_L(K_4) = dim_R(K_4) - 1 = 0 < dim_H(K_4) = dim_P(O_4)-1 = 1$, which implies that $K_4$ is a fractal. Assume further that $G \ne K_4$.  Then obviously, since  $dim_L(G) \leq 2$, $G$ is a fractal if and only if one of the following conditions hold:
% \begin{itemize}
% \item [(a)] $dim_L(G) = 1$ and $dim_H(G) = 2$ (i.e. $dim_R(G) = 2$ and $dim_P(\overline{G}) = 3$)
% \item [(b)] $dim_H(G) \geq 3$ (i.e. $dim_P(\overline{G}) \geq 4$).
% \end{itemize}

First we will prove (1). By Theorem \ref{thm:covervsinter}, $dim_L(G) = dim_R(G) - 1 = 1$ if and only if $G$ is a line graph of a multigraph. Such graphs are characterized by a list of 7 forbidden induced subgraphs, only one of which ($K_{1,3}$) has  the maximal degree, which does not exceed 3 \cite{berge1984hypergraphs}. Therefore $dim_R(G) = 2$ if and only if $G$ is claw-free. 
By Theorem \ref{thm:pdimcharact}, $dim_H(G) = dim_P(\overline{G}) - 1 = 1$ if and only if $G$ is a line graph of a bipartite graph. These graphs are exactly (claw,diamond,odd-hole)-free graphs \cite{brandstadt1999graph}. By combining these facts, we get that $dim_R(G) = 2$, $dim_P(\overline{G}) = 3$ if and only if $G$ is claw-free, but contains diamond or odd hole.

Now we will prove (2). Suppose that $G$ contains the claw or, synonymously, some vertex of $G$ does not belong to a triangle. In this case $\dim_L(G) = 2$, $\Delta(G_{-3}) = 3$ and therefore $G_{-3}$ is either 3- or 4-edge colorable. We will demonstrate that $dim_H(G) = dim_P(\overline{G}) - 1 = 2$ if and only if $G_{-3}$ is of class 1. The necessity is obvious, so it remains to prove the sufficiency. Let $\lambda:E(G_{-3})\rightarrow \{1,2,3\}$ be a 3-edge coloring of $G_{-3}$, and $\mathcal{C} = E(G_{-3})$ be the clique cover of $G_{-3}$ with all clusters being single edges, whose colors are set by $\lambda$. The cover $\mathcal{C}$ could be extended to the equivalent $3$-cover of $G$ as follows. Given that $\Delta(G) = 3$, there are 2 possible arrangements between pairs of triangles in $G$.

1) There are two triangles $T_1 = \{a,b,c\}$ and $T_2 = \{b,c,d\}$ which share an edge $bc$. Suppose also that $a\sim e$ and $b\sim f$, $e,f\not\in \{a,b,c,d\}$ (it is possible that $e = f$). If, without loss of generality $\lambda(\{a,e\}) = \lambda(\{d,f\}) = 1$, then the edges of $T_1$ and $T_2$ could be covered by single-edge cliques, whose colors could be set as $\lambda(\{b,c\}) = 1$, $\lambda(\{a,b\}) = \lambda(\{c,d\}) = 1$, $\lambda(\{a,b\}) = \lambda(\{c,d\}) = 2$, $\lambda(\{a,c\}) = \lambda(\{b,d\}) = 3$. If, say, $\lambda(\{a,e\})  = 1$, $\lambda(\{d,f\}) = 2$, then we may cover $T_1$ and $T_2$ by the cliques $\{a,b\}$, $\{a,c\}$, $\{b,c,d\}$ with colors $\lambda(\{a,b\}) = 2$, $\lambda(\{a,c\}) = 3$, $\lambda(\{b,c,d\}) = 1$.

2) A triangle $T = \{a,b,c\}$ does not share edges with other triangles. Suppose that $a\sim d$, $b\sim e$, $c\sim f$, $d,e,f\not\in \{a,b,c\}$. All these vertices are distinct, and $ad,be$ and $df$ do not belong to any triangles and therefore are present in $G_{-3}$. If the colors $\lambda(\{a,d\}),\lambda(\{b,e\})$ and $\lambda(\{d,f\})$ are distinct, then cover $T$ by single-edge cliques with colors $\lambda(\{a,b\}) = \lambda(\{c,f\})$, $\lambda(\{b,c\}) = \lambda(\{a,d\})$, $\lambda(\{a,c\}) = \lambda(\{b,e\})$. If, alternatively, some of these colors are identical, then there is a color $i\in \{1,2,3\}$ not present among them. In this case cover $T$ with the single clique $\{a,b,c\}$ of the color $\lambda(\{a,b,c\}) = i$.

If in the resulting cover some vertex $v$ is covered by a single triangle $T$, add the single-vertex clique $\{v\}$ with an appropriate color $\lambda(\{v\}) \ne \lambda(T)$. Thus, the constructed cover is a separating equivalent $k$-cover of $G$. This concludes the proof.
\end{proof}

Theorem \ref{thm:fractaldelta3} states that subcubic 1-fractals are reducible to the diamond and odd cycles, while subcubic 2-fractals could be reduced to class 2 graphs. The next theorem will demonstrate, that cubic 2-fractals could be reduced to snarks.

Let $G$ be a cubic graph with $\dim_L(G) = 2$. In this case every vertex of $G_{-3}$ has the degree $0$, $1$ or $3$. Vertices of degree 1 are further referred to as {\it pendant vertices}, and edges incident to pendant vertices as {\it pendant edges}. We will establish the deeper relation between the topology of general cubic fractals and snarks. By Theorem \ref{thm:fractaldelta3}, the case of 1-fractals is rather simple, so we will concentrate on 2-fractals. Thus, we will assume that $G$ contains a claw. Consider the following graph operations:

\begin{itemize}
    \item[O1)] {\it Pendant triple contraction} consists in replacement of pendant vertices $u$, $v$ and $w$ by a single vertex $x$, which is adjacent to all neighbors of $u$, $v$ and $w$ (Fig. \ref{fig:petersen}).
    \item[O2)] {\it Pendant edge identification} of two edges $u_1v_1$ and $u_2v_2$ with $\deg(v_1) = \deg(v_2) = 1$ and $u_1\ne u_2$ consists in removal of $v_1$ and $v_2$ and  replacement of $u_1v_1$ and $u_2v_2$ with the edge $u_1u_2$ (Fig. \ref{fig:petersen}).
\end{itemize}

%  \begin{figure}%
%     \centering
%     \subfloat{{\includegraphics[width=5cm]{triple.png} }}%
%     \qquad
%     \subfloat{{\includegraphics[width=5cm]{edge.png} }}%
%     % \setlength{\abovecaptionskip}{0pt}
%     % \setlength{\belowcaptionskip}{-12pt}
%     \caption{\footnotesize Pendant triple contraction (left) and pendant edge identification (right) operations. Identified vertices and edges are highlighted in red.}%
%     \label{fig:oper}%
% \end{figure}

% \begin{figure}[h]
% \includegraphics[width=.35\textwidth]{triple.png} \hfill
% \includegraphics[width=.35\textwidth]{edge.png} 
% \caption{Pendant triple contraction (left) and pendant edge identification (right) operations. Identified vertices and edges are highlighted in red.}
% \label{fig:oper}
% \end{figure}

Let $G'_{-3}$ be the graph obtained from $G_{-3}$ by removal of isolated vertices and edges. $G_{-3}$ is of class 1 if and only if so is $G'_{-3}$.

\begin{lemma}\label{lemma:pendcolor}
Suppose that $G'_{-3}$ is of class 1. Let $\lambda$ be its 3-edge coloring and $p_1$,$p_2$ and $p_3$ be the numbers of pendant edges with colors 1,2 and 3, respectively. Then $p_1$,$p_2$, $p_3$ are either all odd or all even.
\end{lemma}

\begin{proof}
Let $p_i = 2\alpha_i + \beta_i$, where $\beta_i = p_i\mod 2$.  For each color $i$, consider $\alpha_i$ pairs of $i$-colored pendant edges, identify the edges from each pair, and assign to each newly added edge the color $i$. So, the resulting graph $G''_{-3}$ is also of class 1, has all vertex degrees equal to 1 or 3 and contains $\beta_i$ pendant edges of color $i$. 

Now consider a subgraph $H_{i,j}$ of $G''_{-3}$ formed by edges of colors $i$ and $j$. Obviously, $H_{i,j}$ is a disjoint union of even cycles and, possibly, a single path with distinctly colored end-edges. If the path is not present, then $\beta_i = \beta_j = 0$, otherwise $\beta_i = \beta_j = 1$.
\end{proof}

\begin{theorem}\label{thm:fraccubic}
  The cubic graph $G$ is 2-fractal if and only if it contains a claw and any cubic graph obtained from $G_{-3}'$ by pendant edge identifications and pendant triple contractions either has a bridge or is a snark.
\end{theorem}

\begin{proof}
It can be easily shown that if a cubic graph has a bridge, then it is of class 2 \cite{chladny2010factorisation}. 
Thus, the statement of the theorem is equivalent to the following statement: the cubic graph $G$, which contains a claw, is not 2-fractal if and only if it is possible to construct a cubic graph of class 1  from $G_{-3}'$ by pendant edge identifications and pendant triple contractions. 

To prove the necessity, suppose that $G$ is not 2-fractal, e.g. the graph  $G'_{-3}$ is of class 1. Consider any 3-edge-coloring of $G'_{-3}$ and identify pendant edges of the same color, as described in Lemma \ref{lemma:pendcolor}. If after this operation all pendant edges are eliminated, then the desired graph $H$ of class 1 is constructed. Otherwise, by Lemma \ref{lemma:pendcolor} $H$ contains 3 pendant edges of pairwise distinct colors. Then the desired graph can be constructed by contracting  the pendant end-vertices of these edges.

Conversely, suppose that the graph $H$ of class 1 is obtained from $G_{-3}'$ by pendant edge identifications and pendant triple contractions. Consider any 3-edge-coloring $\lambda$ of $H$. Obviously, $\lambda$ could be transformed into a 3-edge coloring of $G_{-3}'$ by assigning the color $\lambda(u_2v_2)$ to the identified edges $u_1v_1$ and $u_2v_2$.   
\end{proof}
\vspace{-2ex}
Thus, Theorem \ref{thm:fraccubic} states that the biconnected cubic graph is 2-fractal whenever any sequence of removal of triangle edges, isolated edges and vertices, pendant triple contractions and pendant edge identifications, which preserve the graph connectivity, transforms it into a snark. Fig. \ref{fig:petersen} provides an example of such transformations.

\begin{figure}[h]
\begin{center}
\includegraphics[width=\textwidth] {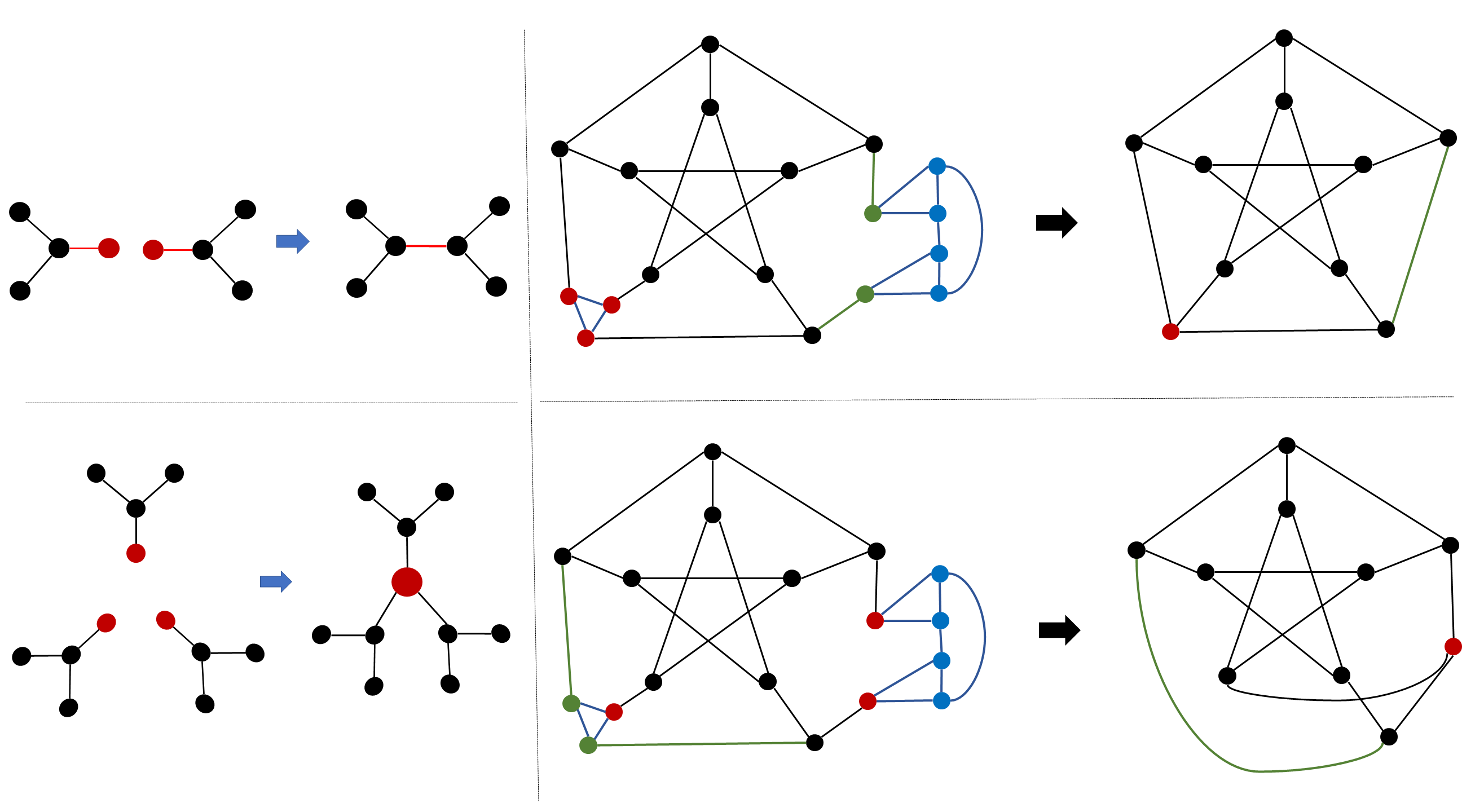}
\caption{\label{fig:petersen} {\footnotesize Left: Pendant triple contraction (top) and pendant edge identification (bottom) operations. Identified vertices and edges are highlighted in red. Right: Transformations of a cubic graph $G$. For each transformation, removed edges are highlighted in blue, vertices involved in the pendant triple contraction are highlighted in red, and vertices and edges involved in the pendant edge identification are highlighted in green. The top transformation converts $G$ into the Petersen snark, which is a fractal. However, $G$ is not fractal, since the bottom transformation converts it into 3-edge-colorable cubic graph.}}
\end{center}
\end{figure}

\subsubsection*{Sierpinski gasket graphs}
Sierpinski gasket graphs $S_n$ \cite{klavvzar2013hamming}
are associated with the Sierpinski gasket - well-known topological fractal with a Hausdorff dimension $\log(3)/\log(2)\approx 1.585$. Edges of $S_n$ are line segments of the $n$-th approximation of the Sierpinski gasket, and vertices are intersection points of these segments (Fig. \ref{fig:serpGasket}). 

Formally, Sierpinski gasket graphs can be defined recursively as follows. Consider tetrads $T_n = (S_n,x_1,x_2,x_3)$, where $x_1,x_2,x_3$ are distinct vertices of $S_n$ called {\it contact vertices}. The first Sierpinski gasket graph $S_1$ is a triangle $K_3$ with vertices $x_1,x_2,x_3$, the first tetrad is defined as $T_1 = (S_1,x_1,x_2,x_3)$. The $(n+1)$-th Sierpinski gasket graph $S_{n+1}$ is constructed from 3 disjoint copies $(S_n,x_1,x_2,x_3)$, $(S'_n,x'_1,x'_2,x'_3)$,  $(S''_n,x''_1,x''_2,x''_3)$ of $n$-th tetrad $T_n$ by gluing together $x_2$ with $x'_1$, $x'_3$ with $x''_2$ and $x_3$ with $x''_1$; the corresponding $(n+1)$-th tetrad is $T_{n+1} = (S_{n+1},x_1,x'_2,x''_3)$.  

\begin{theorem}\label{serpgasket}
For every $n\geq 2$ Sierpinski gasket graph $S_n$ is a fractal with $dim_L(S_n) = 1$ and $dim_H(S_n) = 2$
\end{theorem}

\begin{proof}
First, we will prove that 

\begin{equation}\label{formula:dimssn}
dim_R(S_n)\leq 2, dim_P(\overline{S_n})\leq 3.
\end{equation}

\noindent
for every $n\geq 2$. We will show it using an induction by $n$. In fact, we will prove slightly stronger fact: for any $n\geq 2$ there exists a clique cover $\mathcal{C} = \{C_1,...,C_m\}$ such that (i) every non-contact vertex is covered by two cliques from $\mathcal{C}$; (ii) every contact vertex is covered by one clique from $\mathcal{C}$; (iii) cliques from $\mathcal{C}$ can be colored using 3 colors in such a way, that intersecting cliques receive different colors and cliques containing different contact vertices also receive different colors; (iv) every two distinct vertices are separated by some clique from $\mathcal{C}$.

For $n=2$ the clique cover $\mathcal{C}$ consisting of 3 cliques, that contain contact vertices, obviously satisfies conditions (i)-(iv). Now suppose that $\mathcal{C}$, $\mathcal{C'}$ and $\mathcal{C''}$ are clique covers of $S_n$, $S'_n$ and $S''_n$ with properties (i)-(iv). Assume that $x_i\in C_i$, $x'_i\in C'_i$. $x''_i\in C''_i$ and $C_i,C'_i,C''_i$ have colors $i$, $i=1,...,3$. Then it is straightforward to check, that $\mathcal{C}\cup \mathcal{C'}\cup \mathcal{C''}$ with all cliques keeping their colors is a clique cover of $S_{n+1}$ that satisfies (i)-(iv). So, (\ref{formula:dimssn}) is proved.

Finally, note that for every $n\geq 2$ the graph $S_n$ contains graphs $K_{1,2}$ and $K_4 - e$ as induced subgraphs. Since $dim_R(K_{1,2}) = 2$ and $dim_P(\overline{K_4 - e}) = 3$ (the latter is easy to see using clique cover definition), we have equalities in (\ref{formula:dimssn})
\end{proof}

\begin{figure}[h]
\begin{center}
\includegraphics[width=\textwidth,height=\textheight,keepaspectratio] {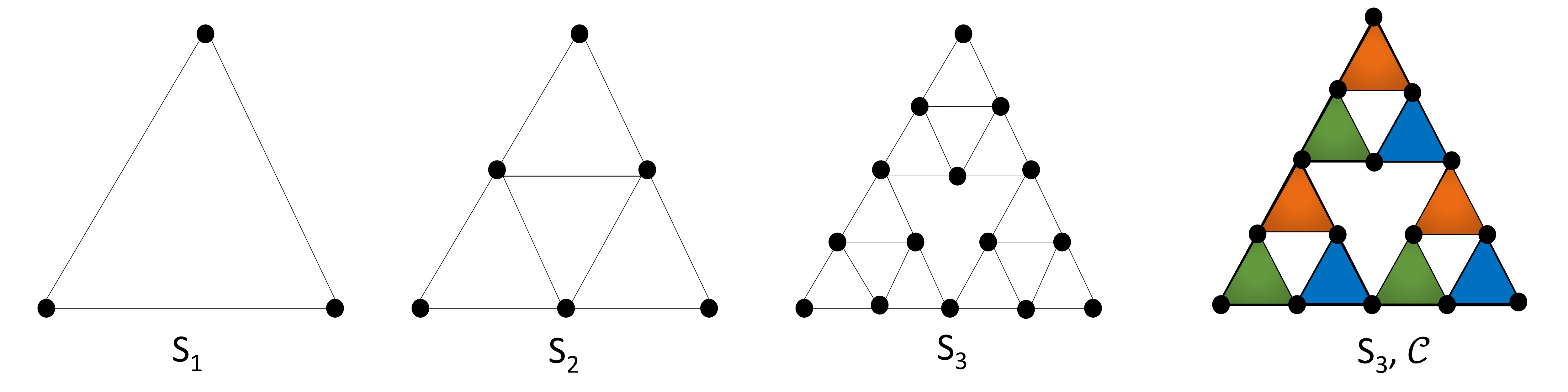}
\caption{\label{fig:serpGasket} {\small Sierpinski gasket graphs $S_1-S_3$ and the optimal equivalent separating $3$-cover of $S_3$. Clusters of the same color are highlighted in red, green and blue. $S_3$ is a fractal: every vertex is covered by 2 clusters, while the clusters can be colored using 3 colors.}}
\end{center}
\end{figure}

% Note that Lebesgue and Hausdorff dimensions of Sierpinski gasket graphs agree with the corresponding dimensions of Sierpinski gasket fractal. 
% As a contrast, note that rectangular grid graphs (cartesian products of 2 paths) are not fractals, just like rectangles are not fractals in $\mathbb{R}^2$. It follows from the fact that they are bipartite.

\section{Information-theoretic connections}\label{sec:inf}

Theorems \ref{thm:covervsinter} and \ref{thm:pdimcharact} allow to interpret graph Lebesgue and Hausdorff dimensions and fractality from the information-theoretical point of view. Indeed, graphs $G$ with $dim_L(G) = k$ could be described by assigning to every vertex $k-1$ a {\it set} of integer "coordinates" represented by hyperedges of a $k$-uniform hypergraph $H$ such that $G=L(H)$. Importantly, these coordinates are non-ordered, and edges of $G$ are defined by a presence of a shared coordinate for their end vertices. In contrast, graphs with $dim_H(G) = k$ are defined by ordered {\it vectors} of coordinates (Theorem \ref{thm:pdimcharact},4)), and an adjacency of a pair of vertices is determined by a presence of a shared coordinate on the same position. Thus, non-fractal graphs are the graphs for which the set and vector representations are equivalent, while fractal graphs have additional structural properties that manifest themselves in extra dimensions needed to describe them using a vector representation.

Relations between graph dimension and information complexity could be analyzed using a {\it Kolmogorov complexity}. Informally, Kolmogorov complexity of a string $s$ could be described as a length of its shortest lossless encoding. 
Formally, let $\mathbb{B}^*$ be the set of all finite binary strings and $\Phi:\mathbb{B}^* \rightarrow \mathbb{B}^*$ be a computable function. {\it Kolmogorov complexity} $K(s) = K_{\Phi}(s)$ of a binary string $s$ with respect to $\Phi$ is the minimal length of a string $s'$ such as $\Phi(s') = s$. Since Kolmogorov complexities with respect to any two functions differ by an additive constant \cite{li2009Kolmogorov}, it can be assumed that some canonical function $\Phi$ is fixed. For two strings $s,t\in \mathbb{B}^*$, a {\it conditional Kolmogorov complexity} $K(s|t)$ is a a length of a shortest encoding of $s$, if $t$ is known in advance. 

Every connected graph $G$ can be encoded using the string representation of an upper triangle of its adjacency matrix. Kolmogorov complexity $K(G)$ of a graph $G$ could be defined as a Kolmogorov complexity of that string \cite{mowshowitz2012entropy}. In addition, the conditional graph Kolmogorov complexity $K(G|n)$ is often considered, which is the complexity given that the number of vertices is known.   Obviously, $K(G) = O(n^2)$ and $K(G|n) = O(n^2)$.  Alternatively, $n$-vertex connected labeled graph can be represented as a list of edges with ends of each edge encoded using their binary representations concatenated with a binary representation of $n$. It gives estimations $K(G) \leq 2m\log(n) + \log(n) = O(m\log(n))$, $K(G|n) \leq 2m\log(n) = O(m\log(n))$ \cite{li2009Kolmogorov,mowshowitz2012entropy}.

Let $dim_H(G) = dim_P(\overline{G}) - 1 = d-1$ and $\mathcal{H}^d(G) = h$. Then $\overline{G}$ is an induced subgraph of a product 
\begin{equation}\label{graphembedmin}
K_{p_1}\times\dots \times K_{p_d},
\end{equation}

\noindent
where $h = p_1\cdot...\cdot p_d$. Thus, by Theorem \ref{thm:pdimcharact}, $G$ and $\overline{G}$ could be encoded using a collection of vectors $\phi(v) = (\phi_1(v),\dots,\phi_d(v))$, $v\in V(G)$, $\phi_j(v)\in [p_j]$. Such encoding could be stored as a string containing binary representations of coordinates $\phi_j(v)$ using $\log(p_j)$ bits concatenated with a binary representations of $n$ and $p_j$, $j=1,...,n$. The length of this string is $(n+1)\sum_{j=1}^d \log(p_j) + \log(n)$. Analogously, if $n$ and $p_j$ are given, then the length of encoding is $n\sum_{j=1}^d \log(p_j)$. Thus, the following estimations hold:

\begin{proposition}
\begin{equation}\label{eq:KolG}
K(G)\leq (n+1)\log(\mathcal{H}^d(G)) + \log(n)
\end{equation}
\vspace{-2ex}
\begin{equation}\label{eq:KolcondG}
K(G|n,p_1,...,p_d)\leq n\log(\mathcal{H}^d(G))
\end{equation}

\end{proposition}

Let $p^* = \max_j p_j$. Then we have $K(G)\leq (n+1)d\log(p^*) + \log(n),$ $K(G|n,p_1,...,p_d)\leq nd\log(p^*).$ By minimality of the representation (\ref{graphembedmin}), we have $p^* \leq n$. Thus $K(G) = O(dn\log(n))$, $K(G|n,p_1,...,p_d) = O(dn\log(n))$. So, Hausdorff (Prague) dimension could be considered as a measure of descriptive complexity of a  graph. 

% In particular, for graphs with a small Hausdorff dimension, the relations (\ref{eq:KolG})-(\ref{eq:KolcondG}) give better estimation of their Kolmogorov complexity than the standard estimations mentioned above.

Relations between Hausdorff (Prague) dimension and Kolmogorov complexity could be used to derive lower bound for Hausdorff dimension of a sparse  Erdős–Rényi random graph. Formally, let $X$ be a graph property and $\mathcal{P}_n(X)$ be the set of labeled n-vertex graphs having this property. The property $X$ holds for {\it almost all graphs} \cite{erdHos1977chromatic}, if $|\mathcal{P}_n(X)|/2^{\binom{n}{2}} \rightarrow 1$ as $n \rightarrow \infty$, i.e. the probability that the sparse  Erdős–Rényi random graph $G(n,\frac{1}{2})$ has the property $X$ converges to 1 as $n \rightarrow \infty$.  We will use the following lemma:

% \begin{lemma}\label{lem:kolcomplb}\cite{buhrman1999kolmogorov}
%  Let $A\subseteq \mathbb{B}^*$ be a finite set of binary strings, $y\in \mathbb{B}^*$ be a fixed string and $c > 0$ be an integer. Then $A$ has at least $|A|(1-2^{-c})+1$ elements $x$ such that $K(x|y) \geq \floor{\log(|A|)} - c$ 
% \end{lemma}

% Following \cite{buhrman1999kolmogorov}, let $|A|$ be the set of labeled $n$-vertex graphs encoded using string representations of their adjacency matrices, $y$ be an empty string and $c = \delta(n)$ be a positive-valued function such that $\delta(n) \rightarrow \infty$ as $n \rightarrow \infty$. The Lemma \ref{lem:kolcomplb} imply the following

\begin{lemma}\label{lem:kolcomplbgr}\cite{buhrman1999kolmogorov} For every $n > 0$ and $\delta: \mathbb{N}\rightarrow \mathbb{N}$, there are at least $2^{\binom{n}{2}}(1-2^{-\delta(n)})$ $n$-vertex labeled graphs $G$ such that $K(G|n) \geq \frac{n(n-1)}{2} - \delta(n)$.
\end{lemma}

The following theorem states that almost all sparse Erdős–Rényi graphs have large Hausdorff dimension:

\begin{theorem}\label{thm:almostalldim}
For every $\epsilon > 0$, almost all sparse Erdős–Rényi graphs have Hausdorff dimension such that

\begin{equation}\label{lowbounddim}
 \frac{1}{1+\epsilon}\Big(\frac{n-1}{2\log(n)} - \frac{1}{n}\Big) - 1 \leq dim_H(G) \leq C\frac{n\log\log (n)}{\log (n)},
\end{equation}

\noindent where $C$ is a constant.
\end{theorem}
\vspace{-4ex}
\begin{proof}
The upper bound has been proved in \cite{cooper2010product}, so we will prove the lower bound.
Let $n_{\epsilon} = \ceil{\frac{2}{\epsilon}}$. Consider a graph $G$ with $n\geq n_{\epsilon}$. From (\ref{eq:KolcondG}) we have $K(G)\leq (n+1)d\log(n) + \log(n)$. Using the fact, that $\frac{1}{n} + \frac{1}{nd}\leq \frac{2}{n}\leq \epsilon$, it is  straightforward to check that $(n+1)d\log(n) + \log(n) \leq (1+\epsilon)nd\log(n)$. Therefore we have

\begin{equation}\label{eq:kolmepsilon}
K(G)\leq (1+\epsilon)nd\log(n)
\end{equation}

Let $X$ be the set of all graphs $G$ such that 

\begin{equation}\label{eq:kolmcondlog}
K(G|n) \geq \frac{n(n-1)}{2} - \log(n)
\end{equation}

Using Lemma \ref{lem:kolcomplbgr} with $\delta(n) = \log(n)$, we conclude that  $|\mathcal{P}_n(X)|/2^{\binom{n}{2}} \geq 1-\frac{1}{n}$, and so almost all graphs have the property $X$.

Now it is easy to see that for graphs with the property $X$ and with $n\geq n_{\epsilon}$ the inequality (\ref{lowbounddim}) holds. It follows by combining inequalities (\ref{eq:kolmepsilon})-(\ref{eq:kolmcondlog}) and using the fact that $K(G|n)\leq K(G)$. 
\end{proof}

\section{Fractality and self-similarity of networks: experimental study}\label{fracgraph:exper}

\subsubsection*{Calculation of Lebesgue and Hausdorff dimensions} The problems of calculating Hausdorff and Lebesgue dimension of graphs are algorithmically hard. Indeed, the problem of verifying whether $\dim_L(G) \leq k$ is NP-complete for $k > 2$ \cite{poljak1981complexity} (the complexity for  $k = 2$ is unknown).
It is easy to see that the problems of checking whether $\dim_L(G) \leq k$  and  deciding whether a given graph is a fractal are also NP-complete. It follows from Proposition \ref{prop:class2} and NP-completness of the edge chromatic number problem for triangle-free cubic graphs \cite{holyer1981np}. Therefore we use Integer Linear Programming (ILP) for calculation of Hausdorff and Lebesgue dimensions and detection of fractal graphs. Let us call a clique $l$-cover and a separating equivalent $h$-cover {\it optimal}, if $l=\dim_L(G)+1$ and $h=\dim_H(G)+1$.  For Lebesgue dimension, we are looking for an optimal clique cover which consists of a minimal number of clusters. For such cover, every maximal clique of $G$ contains at most one cluster (otherwise, we can join the clusters contained in the same clique). Using this fact, we proceed as follows. Let $\{C_1,...,C_q\}$ be the list of maximal cliques of $G$  found using Bron–Kerbosch algorithm \cite{bron1973algorithm}. Then an optimal clique cover is found by solving the following ILP problem:

% In light of these facts, we adopt the following partially heuristic approach for calculation of Hausdorff and Lebesgue dimensions and detection of fractal graphs. Let us call a clique $l$-cover and a separating equivalent $h$-cover {\it optimal}, if $k=\dim_L(G)-1$ and $l=\dim_H(G)-1$.  For Lebesgue dimension, we construct the corresponding optimal clique cover as follows. First we note that $dim_L(G) \geq s(G)-1$, where $s = s(G)$ is the maximum integer such that $G$ contains an induced star $K_{1,r}$. Thus, obviously if $G$ has a clique $s$-cover, then it is optimal. Next, we restrict ourselves to clique covers which consist only of maximal (by inclusion) cliques. Both the value of $r(G)$ and the list of maximal cliques $\{C_1,...,C_q\}$ are found using Bron–Kerbosch algorithm \cite{bron1973algorithm} applied to a neighborhood of each vertex and to the whole graph, respectively. Then a clique cover is found by solving the following Integer Linear Programming problem:

\begin{equation}\label{eq:ilplobj}
    z \rightarrow \min
\end{equation}
\vspace{-2ex}
\begin{equation}\label{eq:ilpl1}
  \sum_{j=1}^q x_{i,j} \leq z,\;\;\; i=1,...,n.  
\end{equation}
\vspace{-1ex}
\begin{equation}\label{eq:ilpl2}
  \sum_{j=1}^q y_{e,j} \geq 1,\;\;\; e\in E(G).  
\end{equation}
\vspace{-1ex}
\begin{equation}\label{eq:ilpl3}
  y_{e,j} \leq  x_{u,j}, y_{e,j} \leq  x_{v,j}, y_{e,j} \geq  x_{u,j} + x_{v,j}-1,\;\;\;e = uv\in E(G), j=1,...,q.
\end{equation}
\vspace{-1ex}
\begin{equation}\label{eq:ilplbounds}
    0 \leq x_{i,j} \leq a_{i,j};  0 \leq y_{e,j} \leq b_{e,j};  1 \leq z \leq \Delta,\;\; i=1,...,n, e\in E(G), j=1,...,q.
\end{equation}

Here $z$ is the variable representing the rank dimension of $G$;  the binary variables $x_{i,j}$ and $y_{e,j}$ indicate whether a vertex $i$ and an edge $e$ are covered by a cluster contained in $C_j$; $a_{i,j}$ and $b_{e,j}$ are binary constants indicating whether a corresponding vertex/edge belongs to $C_j$ and $\Delta = \Delta(G)$. The constraints (\ref{eq:ilpl1}) state that every vertex is covered by at most $z$ cliques;  the constraints (\ref{eq:ilpl2}) enforce the requirement that every edge is covered by at least one clique and the constraints (\ref{eq:ilpl3}) ensure that an edge is covered by a clique if and only if both its ends are covered by it.

 As before, we assume that a given graph has no true twins. In this case the Hausdorff dimension of the graph is found by generating the set $\mathcal{C} = \{C_1,...,C_Q\}$ of all cliques of $G$ and solving the following ILP problem:

\begin{equation}\label{eq:ilpobjh}
    \sum_{k=1}^K x_{i,k} \rightarrow \min
\end{equation}
\vspace{-3ex}
\begin{equation}\label{eq:ilph1}
    \sum_{i=1}^Q x_{i,k} \leq Qy_k,\;\;\; k=1,...,K
\end{equation}
\vspace{-1ex}
\begin{equation}\label{eq:ilph3}
    \sum_{k=1}^K x_{i,k} \leq 1,\;\;\; i=1,...,Q
\end{equation}
\vspace{-1ex}
\begin{equation}\label{eq:ilph2}
    \sum_{i: v\in C_i} x_{i,k} \leq 1,\;\;\; k=1,...,K; v\in V(G)
\end{equation}
\vspace{-3ex}
\begin{equation}\label{eq:ilph4}
    \sum_{i: u,v\in C_i} \sum_{k=1}^K x_{i,k} \geq 1,\;\;\; uv\in E(G)
\end{equation}

\noindent Here $K = \Delta(G) + 1$ is an upper bound on the Hausdorff dimension of the graph $G$. The binary variable $x_{i,k}$ indicates whether the clique $C_i$ is colored by a color $k$, and the binary variable $y_k$ indicates whether the color $k$ is used; the relation between these variables is enforced by the constraints (\ref{eq:ilph1}).  The constraints (\ref{eq:ilph2}) state that every clique receives at most one color; it is possible that a clique does not have any color, which means that a clique is not selected as a cluster. By the constraints (\ref{eq:ilph2}), all cliques containing any given vertex $v$ receive different colors, and the constraints (\ref{eq:ilph4}) ensure that at least one of the cliques covering any edge $uv$ receives a color (i.e. selected as a cluster). If the Lebesgue dimension has been previously estimated, then the calculations could be accelerated by removal from $\mathcal{C}$ of all cliques that intersect at most $dim_L(G)$ other cliques.  

For all networks described below  Lebesgue and Hausdorff dimensions were calculated using Gurobi 8.1.1. %\cite{gurobi}.

\begin{figure}%
    \centering
    \subfloat{{\includegraphics[width=4cm]{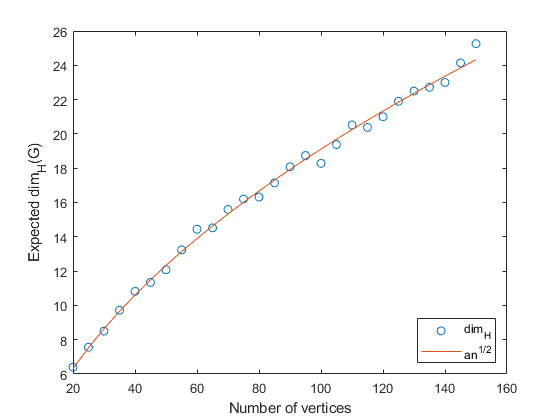} }}%
    \qquad
    \subfloat{{\includegraphics[width=4cm]{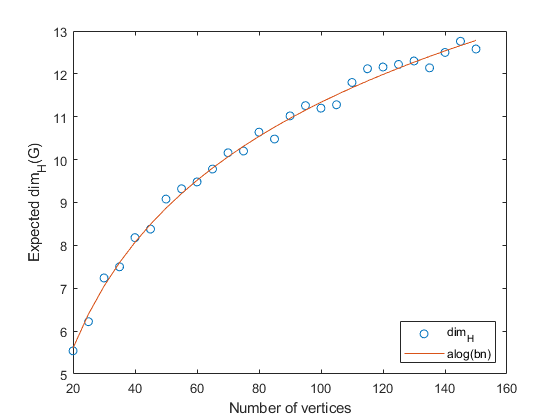} }}%
    \qquad
    \subfloat{{\includegraphics[width=4cm]{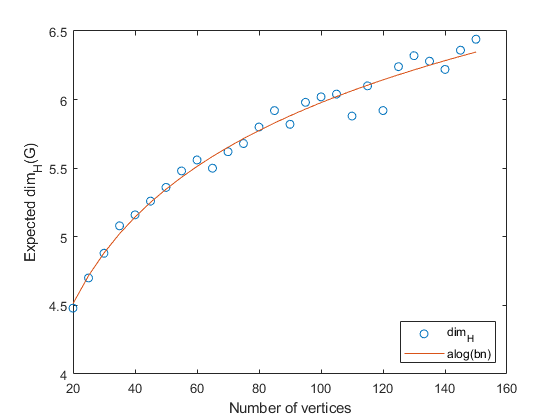} }}%
    \qquad
    \subfloat{{\includegraphics[width=4cm]{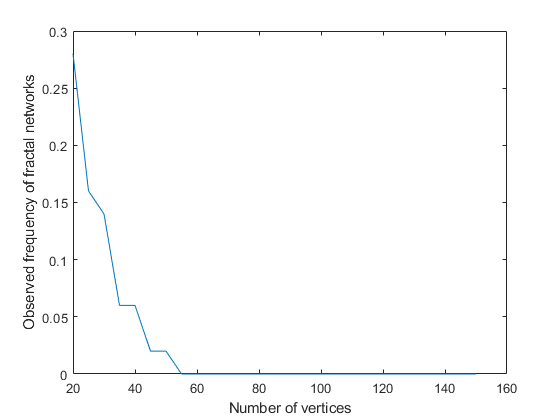} }}%
    \qquad
    \subfloat{{\includegraphics[width=4cm]{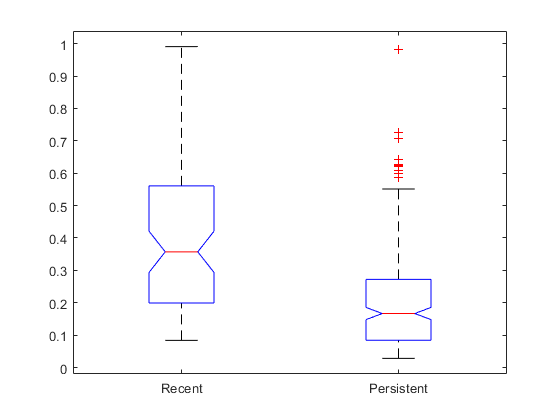} }}%
    \caption{\footnotesize Top: expected Hausdorff dimensions for Preferential Attachment (left), Erd\"{o}s-Renyi (center) and Watts-Strogatz (right) networks. Bottom: (left) observed frequency of fractal networks for Wattz-Strogatz model; (right): distributions of normalized Hausdorff dimensions for genetic networks of recent and persistent intra-host HCV populations.}%
    \label{fig:experModel}%
\end{figure}

\subsubsection*{Network models} Three common models have been considered: preferential attachment, Erd\"{o}s-Renyi and Watts-Strogatz. For each model, 1350 networks with 20-150 vertices have been generated using MIT Matlab Toolbox for Network Analysis \cite{bounova2012overview}. For a given network size, the model parameters were selected in a way resulting in the same network density for all three models. 
 
 For preferential attachment and Erd\"{o}s-Renyi networks, their average Hausdorff dimensions grew as $\Theta(n^{1/2})$ ($R^2 = 0.9958$) and $\Theta(\log(n))$ ($R^2 = 0.9943$), respectively, just as suggested by the estimations in Section \ref{sec:fractheory} (Fig. \ref{fig:experModel}). Hausdorff dimension of Watts-Strogatz networks showed the behaviour similar to that of the latter ($R^2 = 0.9789$).  Importantly, none of the analyzed preferential attachment and Erd\"{o}s-Renyi networks were fractal. In contrast, Watts-Strogatz fractal networks have been observed, although their proportion exponentially decreases with the growth of $n$ (Fig. \ref{fig:experModel}). It suggests, that for the analyzed models the network fractality is rare. It is known that almost all graphs (in the sense of Erd\"{o}s-Renyi graphs $G(n,\frac{1}{2})$) are of class 1 \cite{erdHos1977chromatic}. Thus graph fractality inherits the asymptotic behaviour of edge colorings dichotomy.

\subsubsection*{Real networks with known communities} To calculate Lebesgue and Hausdorff dimensions of a graph $G$, it is required to find the sets of communities of $G$ representing clusters of its optimal $l$-cover and equivalent separating $h$-cover. If the communities $\Omega$ are known in advance, we may consider restricted Lebesgue dimension $dim_L(G,\Omega)$ and restricted Hausdorff dimension $dim_H(G,\Omega)$ with respect to these communities that can be defined as follows: given a hypergraph $H = (V(G),\Omega)$ with all twin vertices removed, $dim_L(G,\Omega) = \Delta(H)$ and $dim_H(G,\Omega) = \chi(L(H))$.
 
 We calculated the restricted dimensions of 8 real-life networks with known ground-truth communities from Stanford Large Network Dataset \cite{snapnets}. To calculate Hausdorff dimensions, the standard ILP formulation for the Vertex Coloring problem has been utilized. If the solver was not able to handle the full community dataset, we analyzed 5000 communities of highest quality provided by the database's curators. 3 out of 8 networks have been found to be fractal. It is significantly higher proportion than suggested by the analysis of network models above, thus suggesting that for real networks the fractality is more prevalent.

\subsubsection*{Viral genetic networks} For a given biological population, the vertices of its genetic network \cite{campo2014next} are genomes of the members of the population, and two vertices are adjacent if and only if the corresponding genomes are genetically close. Genetic network represents a snapshot of the mutational landscape of the population, whose structure is shaped by selection pressures, epistatic interactions and other evolutionary factors \cite{schaper2011epistasis}. 
 
RNA viruses exist in infected hosts as highly heterogeneous populations of genomic variants or {\it quasispecies}. Recently, indications of self-similarity in quasispecies genetic networks were found (D.S. Campo, personal communication). We investigated this phenomenon using the proposed theoretical framework. We considered genetic networks of intra-host Hepatitis C (HCV) populations of $n=355$ infected individuals at early ($n=98$) and persistent ($n=255$) stages of infection \cite{lara2017identification}. The networks were constructed using high-throughput sequencing data of HCV Hypervariable Region 1 (HVR1), with two HVR1 sequences being adjacent, if they differ by a single mutation. For each network, the dimensions of the largest connected component has been calculated with the time limit of $600s$. Solutions have been obtained for $n=323$ networks with $361.71$ vertices in average.
 
%  \begin{figure}%
%     \centering
%     \subfloat{{\includegraphics[width=5cm]{fracWS.png} }}%
%     \qquad
%     \subfloat{{\includegraphics[width=5cm]{RecPersHD.png} }}%
%     % \setlength{\abovecaptionskip}{0pt}
%     % \setlength{\belowcaptionskip}{-12pt}
%     \caption{\footnotesize Left: observed frequency of fractal networks for Wattz-Strogatz model. Right: distributions of normalized Hausdorff dimensions for genetic networks of recent and persistent intra-host HCV populations.}%
%     \label{fig:exper}%
% \end{figure}

The normalized Hausdorff dimensions $\overline{dim}_H(G)$ of networks of persistent populations was found to be significantly lower than for recent populations ($p=3.01\cdot 10^{-13}$, Kruskal-Wallis test, Fig. \ref{fig:experModel}), thus indicating significantly higher level of their self-similarity. This finding is biologically significant. Indeed, one of fundamental questions in the study of pathogens is the role of different evolutionary mechanisms in the infection progression. For HCV, the standard assumption, that the major driving force of intra-host viral evolution is the continuous immune escape, has been put into question by the series of observations that suggest high level of intra-host viral adaptation \cite{campo2014next,gismondi2013dynamic}. Increase in self-similarity of HCV genetic networks implies the gradual self-organization of vital populations and emergence of
structural patterns in population composition and points to the presence of a dynamical mechanism of their formation at later stages of infection, which may be associated with the higher level of adaptation and specialization of viral variants. Thus, it supports the adaptation hypothesis and is consistent with the recently proposed models of viral antigenic cooperation \cite{skums2015antigenic,domingo2019social}, which suggests the emergence of complementary specialization of viral variants and their adaptation to the host environment as a quasi-social system.

\section{Conclusions and future research}

We presented a theoretical framework for study of fractal properties of networks, which is based on the combinatorial and graph-theoretical notions and methods. We anticipate that it could be useful for theoretical studies of properties of network models as well as for analysis of experimental networks which arise in biology, epidemiology and social sciences. In particular, this study has been triggered by biological questions raised by studies of genetic and cross-immunoreactivity networks of RNA viruses, such as HIV and Hepatitis C \cite{campo2014next,skums2015antigenic}. These networks have small-world properties, which impede application of standard methods for the analysis of their fractal characteristics. The ideas presented in this paper may facilitate study of properties of such networks using convergent machineries of graph theory, general topology, algorithmic information theory and discrete optimization. In particular, the problems of detection of fractal properties could be now formulated and studied as rigorously defined algorithmic problems. One such problem is the detection of fractal graphs, another - calculation of invariants measuring how close is the given graph from being a fractal (in terms of number of edges or certain modification operations).  It should be noted, however, that these problems are likely to be algorithmically hard since, as discussed above, the fractality recognition problem is NP-complete. Thus, approximation algorithms and heuristics for these problems should be developed, and the graph classes where the problems become polynomially solvable should be identified. Another important direction of future research is the discovery of structural properties of graph fractals in general and in particular graph classes, as well as identification of network construction models which produce fractals. In particular, our results suggest that fractality is more common for real-life networks than can be concluded from analysis of classical network models.

\section{Funding} This work was partially supported by the National Institutes of Health, grant number 1R01EB025022 "Viral Evolution and Spread of Infectious Diseases in Complex Networks: Big Data Analysis and Modeling".

\bibliographystyle{plain}
\bibliography{fracdim_biblio}

\end{document}